\documentclass[11pt]{amsart}
 
\usepackage{fullpage}
\usepackage{amssymb}
\usepackage{amsmath}
\usepackage{amsxtra}
\usepackage{amscd}
\usepackage{graphicx}
\usepackage{amsfonts}
\usepackage{pb-diagram}
\usepackage{float}
\usepackage{enumerate}
\usepackage{dsfont}
\usepackage{multirow}
\usepackage{color}
\usepackage{rotating}
\usepackage{url}
\usepackage{hyperref}

\newtheorem{thm}{Theorem}
\newtheorem{lemma}{Lemma} 

\newtheorem{defn}{Definition}	
\newtheorem{conj}{Conjecture}

\usepackage{graphicx,color}
\usepackage{pstricks-add}

\numberwithin{equation}{section}

\begin{document}

\title[A potential basis of $\mathrm{H}_{2,n}(q)$]
  {A spanning set and potential basis of the mixed Hecke algebra on two fixed strands}
\author{Dimitrios Kodokostas}
\address{Department of Mathematics,
National Technical University of Athens,
Zografou campus,
{GR-15780} Athens, Greece.}
\email{dkodokostas@math.ntua.gr}

\author{Sofia Lambropoulou}
\address{Department of Mathematics,
National Technical University of Athens,
Zografou campus,
{GR-15780} Athens, Greece.}
\email{sofia@math.ntua.gr}
\urladdr{http://www.math.ntua.gr/$\sim$sofia}

\subjclass[2010]{57M27, 57M25, 20F36, 20C08}
\date{}

\thanks{This research has been co-financed by the European Union (European Social Fund - ESF) and Greek national funds through the Operational Program "Education and Lifelong Learning" of the National Strategic Reference Framework (NSRF) - Research Funding Program: THALES: Reinforcement of the interdisciplinary and/or inter-institutional research and innovation.}

\keywords{mixed braid group on two fixed strands, mixed Hecke algebra, quadratic relation, Hecke-type algebras.}

\begin{abstract}
The mixed braid groups $B_{2,n}, \ n \in \mathbb{N}$, with two fixed strands and $n$ moving ones, are known to be related to the knot theory of certain families of $3$-manifolds. In this paper we define the  mixed Hecke algebra $\mathrm{H}_{2,n}(q)$ as the quotient of the group algebra ${\mathbb Z}\, [q^{\pm 1}] \, B_{2,n}$ over the quadratic relations of the classical Iwahori-Hecke algebra for the braiding generators. 
We furhter provide a potential basis $\Lambda_n$ for $\mathrm{H}_{2,n}(q)$, which we prove is a spanning set for the $\mathbb{Z}[q^{\pm 1}]$-additive structure  of this algebra. The sets $\Lambda_n,\ n \in \mathbb{Z}$ appear to be good candidates for an inductive basis suitable for the construction of Homflypt-type invariants for knots and links in the above $3$-manifolds.
\end{abstract}

\maketitle

\section*{Introduction} \label{introduction}

Knots and links in certain families of 3-manifolds (knot complements in $S^3$, c.c.o. 3-manifolds, handlebodies) can be represented by mixed links and mixed braids in $S^3$ \cite{LR1,HL,LR2,DL}, and the corresponding braid structures in these manifolds are encoded by the mixed braid groups $B_{m,n}$ or appropriate cosets of theirs \cite{La2}. A {\it mixed link} or a {\it mixed braid}  comprises a fixed part (sublink or subbraid, respectively), which represents the 3-manifold we are in, and a moving part (sublink or subbraid), which represents the link or braid, respectively, in that 3-manifold. Furthermore, knot and link isotopy in each one of these 3-manifolds corresponds to appropriate mixed braid equivalences \cite{LR1,HL,LR2,DL}.
 The mixed braid groups have braiding generators and looping generators, see Figure~\ref{mixedgenrs}.

In this paper, we first  recall briefly from \cite{La2} the definition of $B_{2,n}$ and its presentation and we then define the quotient algebra  $\mathrm{H}_{2,n}(q)$ over the quadratic relations of the classical Iwahori--Hecke algebra for the braiding generators. These algebraic structures are  related to the knot theory of the handlebody of genus two,  the complement of the 2-unlink in $S^3$ and the connected sums of two lens spaces. These last spaces are of interest also in biological applications \cite{BM}.   For the algebra  $\mathrm{H}_{2,n}(q)$ we derive a subset $\Lambda_n$, for which we prove that it provides a spanning set for the additive structure of the algebra and also a potential linear basis, a  subject of sequel  research.

We are interested in the sets $\Lambda_n, \ n\in \mathbb{N}$, since they appear to be  appropriate  inductive bases for the  algebras $\mathrm{H}_{2,n}(q), n \in \mathbb{N}$, in order to construct Homflypt-type invariants for  oriented links in the $3$-manifolds whose braid structure is encoded by the groups $B_{2,n}$. Such invariants have already been constructed for the solid torus \cite{La1} and the lens spaces $L(p,1)$ \cite{DLP} utilizing a  similar inductive basis for the algebra $\mathrm{H}_{1,n}(q)$. The algebra $\mathrm{H}_{1,n}(q)$ is a quotient of the Artin braid group of type $\mathrm{B}$, $B_{1,n}$, and a subalgebra of $\mathrm{H}_{2,n}(q)$. In \cite{KL} we also discuss interesting quotients of the algebra $\mathrm{H}_{2,n}(q)$, which generalize the Iwahori--Hecke algebra of type $\mathrm{B}$ and the cyclotomic Ariki--Koike algebras of type $\mathrm{B}$.

\section{$B_{2,n}$ and $\mathrm{H}_{2,n}(q)$} \label{defns}

The elements of the \textit{mixed braid group} $B_{2,n}$ \textit{on two fixed strands} are the braids with $n+2$ strands with the first two of them being identically straight. The group operation of $B_{2,n}$ is the usual braid concatenation. In terms of generators and relations this group is generated  by the braids $\mathcal{T}, \tau, \sigma_1,\dots, \sigma_n$ shown in Figure \ref{mixedgenrs}, for which the following defining relations hold (cf. \cite{La2}):
\begin{equation*}
\begin{array}{rclclcl} 
  \sigma_k \sigma_j&=&\sigma_j \sigma_k  &  & \mbox{for \quad $|k-j|>1$}\\
  \sigma_k \sigma_{k+1} \sigma_k &=& \sigma_{k+1} \sigma_k \sigma_{k+1} && \mbox{for \quad $ 1 \leq k \leq n-1$}\\
  \mathcal{T}\, \sigma_k &=& \sigma_k \,\mathcal{T} \ &&  \mbox{for \quad $ k \geq 2$}\\
  \tau\, \sigma_k &=& \sigma_k\, \tau  && \mbox{for \quad $ k \geq 2$}\\
    \mathcal{T}\, \sigma_1 \,\mathcal{T}\, \sigma_1 &=& \sigma_1 \,\mathcal{T}\, \sigma_1\, \mathcal{T}  && \\
   \tau\, \sigma_1 \,\tau\, \sigma_1 &=& \sigma_1 \,\tau\, \sigma_1\, \tau  && \\
   \tau (\sigma_1 \mathcal{T} {\sigma_1}) &=&  (\sigma_1 \mathcal{T} {\sigma_1}) \tau  &&
\end{array}
\end{equation*}

We call $\tau,\mathcal{T}$ and their inverses the \textit{looping generators}, and the $\sigma_i$'s and their inverses the \textit{braiding generators}. The braiding generators are the usual crossings between consecutive moving strands. We also call the first two strands which are straight in any element of $B_{2,n}$  \textit{fixed strands} and we denote them $I,II$, while we call the rest of the strands  \textit{moving strands} and we conveniently indicate them as $1,2,\ldots,n$ from left to right.

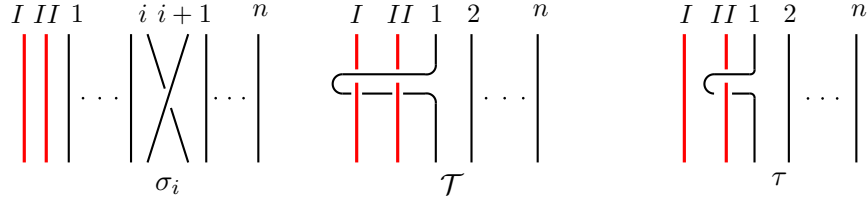
\begin{figure}[!h]
\centering
\psset{xunit=1.0cm,yunit=1.0cm,algebraic=true,dimen=middle,dotstyle=o,dotsize=5pt 0,linewidth=0.8pt,arrowsize=3pt 2,arrowinset=0.25}
\begin{pspicture*}(2.4,1.487965395051425)(13.790397349032817,4.2)
\psaxes[labelFontSize=\scriptstyle,xAxis=true,yAxis=true,Dx=0.5,Dy=0.5,ticksize=-2pt 0,subticks=2]{->}(0,0)(2.477917115593102,1.487965395051425)(13.790397349032817,3.974040998778535)
\psline(16.88064511893745,3.7070798601232746)(16.88064511893745,2.)
\psline[linewidth=1.2pt,linecolor=red](2.872832267578602,3.7070798601232746)(2.872832267578602,2.)
\psline(3.1720411124949095,3.7070798601232746)(3.1720411124949095,2.)
\psline[linewidth=1.2pt,linecolor=red](2.5806877405277304,3.7070798601232746)(2.58068774052773,2.)
\psline(4.,3.7070798601232746)(4.,2.)
\psline(4.7601495970506935,3.7070798601232746)(4.219272069701983,2.)
\psline(4.219272069701984,3.7070798601232746)(4.448902943228507,2.982334941655037)
\psline(4.5305187235241675,2.7247449184682466)(4.7601495970506935,2.)
\psline(5.,3.7070798601232746)(5.,2.)
\psline(5.692300229289959,3.7070798601232746)(5.692300229289959,2.)
\psline(9.406147349086977,3.7070798601232746)(9.406147349086975,2.)
\psline(8.526217751987549,3.7070798601232746)(8.526217751987549,2.)
\psline[linewidth=1.2pt,linecolor=red](7.,3.7070798601232746)(7.,3.237954573315138)
\psline[linewidth=1.2pt,linecolor=red](7.545093461271708,3.7070798601232746)(7.54509346127171,3.237954573315138)
\psline(8.05176137551848,3.7070798601232746)(8.05176137551848,3.306124117223481)
\psline(8.05176137551848,2.783430863866392)(8.05176137551848,2.)
\parametricplot{1.5707963267948966}{4.71238898038469}{1.*0.1283800922678182*cos(t)+0.*0.1283800922678182*sin(t)+6.8090735867741|0.*0.1283800922678182*cos(t)+1.*0.1283800922678182*sin(t)+3.0447774905449365}
\psline(6.8090735867741,3.1731575828127547)(7.908908732089605,3.1731575828127547)
\psline(2.,3.237954573315138)(2.,2.8516004077747352)
\psline[linewidth=1.2pt,linecolor=red](7.,3.0611484297627505)(7.,2.)
\psline[linewidth=1.2pt,linecolor=red](7.54509346127171,3.0611484297627505)(7.545093461271708,2.)
\psline(6.8090735867741,2.9163973982771183)(6.928024863590807,2.9163973982771187)
\psline(7.071975136409195,2.9163973982771187)(7.473118324862518,2.916397398277118)
\psline(7.617068597680906,2.916397398277118)(7.908908732089605,2.9163973982771183)
\parametricplot{-1.7485207621962244}{0.10606967380701743}{1.*0.12197428168953435*cos(t)+0.*0.12197428168953435*sin(t)+7.930472603400021|0.*0.12197428168953435*cos(t)+1.*0.12197428168953435*sin(t)+3.2932105913039447}
\parametricplot{-0.10606967380700993}{1.7485207621962229}{1.*0.12197428168953466*cos(t)+0.*0.12197428168953466*sin(t)+7.930472603400021|0.*0.12197428168953466*cos(t)+1.*0.12197428168953466*sin(t)+2.7963443897859275}
\psline[linewidth=1.2pt,linecolor=red](11.355797037423184,3.7070798601232746)(11.355797037423184,2.)
\psline(12.745179813941085,3.7070798601232746)(12.745179813941085,2.)
\psline(13.64270935391784,3.7070798601232746)(13.642709353917839,2.)
\parametricplot{1.5707963267948966}{4.71238898038469}{1.*0.12838009226781777*cos(t)+0.*0.12838009226781777*sin(t)+11.751031580270006|0.*0.12838009226781777*cos(t)+1.*0.12838009226781777*sin(t)+3.0447774905449316}
\parametricplot{-1.70218102516555}{0.15240941083776374}{1.*0.056680288649645326*cos(t)+0.*0.056680288649645326*sin(t)+12.23271742160917|0.*0.056680288649645326*cos(t)+1.*0.056680288649645326*sin(t)+3.2293493689360995}
\psline(11.751031580270006,3.1731575828127494)(12.22529190521801,3.1731575828127547)
\psline[linewidth=1.2pt,linecolor=red](11.913326513838179,3.7070798601232746)(11.913326513838179,3.237954573315128)
\psline[linewidth=1.2pt,linecolor=red](11.913326513838177,3.0611484297627447)(11.913326513838177,2.)
\psline(12.288740680877106,3.7070798601232746)(12.288740680877106,3.237954573315127)
\psline(11.990257520452031,2.916397398277101)(12.22529190521801,2.9163973982771183)
\parametricplot{-0.15240941083784243}{1.7021810251654024}{1.*0.0566802886496529*cos(t)+0.*0.0566802886496529*sin(t)+12.232717421609163|0.*0.0566802886496529*cos(t)+1.*0.0566802886496529*sin(t)+2.8602056121537633}
\psline(12.288740680877105,2.8516004077747303)(12.288740680877105,2.)
\small
\rput[tl](2.4,4.1){$I$}
\rput[tl](2.7,4.1){$II$}
\rput[tl](3.2,4.1){$1$}
\rput[tl](5.614711044629483,4.1){$n$}
\rput[tl](4.1,4.1){$i$}
\rput[tl](4.35,4.1){$i+1$}

\rput[tl](6.91,4.1){$I$}
\rput[tl](7.4,4.1){$II$}
\rput[tl](11.29,4.1){$I$}
\rput[tl](11.7,4.1){$II$}
\rput[tl](7.975649028452078,4.1){$1$}
\rput[tl](12.2,4.1){$1$}
\rput[tl](9.34,4.1){$n$}
\rput[tl](13.57349138585477,4.1){$n$}
\rput[tl](8.46785871412534,4.1){$2$}
\rput[tl](12.672497384961341,4.1){$2$}
\normalsize
\rput[tl](4.321617802606506,1.8049817472045466){$\sigma _i$}
\rput[tl](8.109129621177031,1.8383518895364541){$\mathcal{T}$}
\rput[tl](12.51398918110046,1.8633794962853847){$\tau$}
\begin{scriptsize}
\psdots[dotsize=3pt 0,dotstyle=*,linecolor=green](2.,2.)
\psdots[dotsize=3pt 0,dotstyle=*,linecolor=green](2.,3.7070798601232746)
\psdots[dotsize=3pt 0,dotstyle=*,linecolor=green](16.88064511893745,3.7070798601232746)
\psdots[dotsize=3pt 0,dotstyle=*,linecolor=darkgray](16.88064511893745,2.)
\psdots[dotsize=3pt 0,dotstyle=*](2.,3.237954573315138)
\psdots[dotsize=3pt 0,dotstyle=*,linecolor=darkgray](16.880645118937455,3.237954573315138)
\psdots[dotsize=1pt 0,dotstyle=*](3.3619484844450165,2.853539930061637)
\psdots[dotsize=1pt 0,dotstyle=*](3.582230490176628,2.853539930061637)
\psdots[dotsize=1pt 0,dotstyle=*](5.1242045302979085,2.853539930061637)
\psdots[dotsize=1pt 0,dotstyle=*](5.309705166703476,2.853539930061637)
\psdots[dotsize=1pt 0,dotstyle=*](3.80251249590824,2.853539930061637)
\psdots[dotsize=1pt 0,dotstyle=*](5.495205803109044,2.853539930061637)
\psdots[dotsize=3pt 0,dotstyle=*](2.,2.8516004077747352)
\psdots[dotsize=3pt 0,dotstyle=*,linecolor=darkgray](16.880645118937455,2.8516004077747352)
\psdots[dotsize=1pt 0,dotstyle=*](8.966182550537262,2.853539930061637)
\psdots[dotsize=1pt 0,dotstyle=*](8.719765223706005,2.853539930061637)
\psdots[dotsize=1pt 0,dotstyle=*](9.212599877368518,2.853539930061637)
\psdots[dotsize=3pt 0,dotstyle=*](2.,3.0611484297627505)
\psdots[dotsize=1pt 0,dotstyle=*](13.193944583929463,2.8516004077747295)
\psdots[dotsize=1pt 0,dotstyle=*](13.,2.851600407774732)
\psdots[dotsize=1pt 0,dotstyle=*](13.387889167858926,2.851600407774727)
\end{scriptsize}
\end{pspicture*}
\caption{The generators of $B_{2,n}$.}
\label{mixedgenrs}
\end{figure}

To what follows, it will be important to define the \textit{looping elements} or just \textit{loopings}  $\mathcal{T}_i,\tau_i$  and  $\mathcal{T}_i^{- 1},\tau_i^{- 1}$  of  $B_{2,n}$ (see Figure \ref{figure_3}), which are mixed braids  with  all strands straight except for the $i$-th moving strand that loops once around $I$ or $II$ respectively and encloses all intermediate strands. Namely, $\mathcal{T}_1 :=\mathcal{T},\ \tau_1:=\tau$ and for $ i>1$: 
\begin{equation} \nonumber
  \mathcal{T}_i := \sigma_{i-1} \ldots \sigma_1\mathcal{T} \sigma_1 \ldots \sigma_{i-1} \quad \mbox{and} \quad   \tau_i := \sigma_{i-1} \ldots \sigma_1\tau \sigma_1 \ldots \sigma_{i-1}.
	\end{equation}
	We say that the loopings   $\mathcal{T}_i^{\pm 1},\tau_i^{\pm 1}$ as well as the braiding generators $\sigma_i^{\pm 1}$ have \textit{index} $i$. Clearly, the braiding generators along with  the looping elements generate $B_{2,n}$.

\begin{figure}[h]
\centering
\psset{xunit=1.0cm,yunit=1.0cm,algebraic=true,dimen=middle,dotstyle=o,dotsize=5pt 0,linewidth=0.8pt,arrowsize=3pt 2,arrowinset=0.25}
\begin{pspicture*}(5.7299729881424675,1.5)(14.140545192001177,4.15)
\psaxes[labelFontSize=\scriptstyle,xAxis=true,yAxis=true,Dx=0.5,Dy=0.5,ticksize=-2pt 0,subticks=2]{->}(0,0)(5.7299729881424675,1.6682978836699438)(14.140545192001177,3.904158293171864)
\psline(16.88064511893745,3.7070798601232746)(16.88064511893745,2.)
\psline(9.010123622070456,3.7070798601232746)(9.010123622070456,2.)
\psline(8.384433343672242,3.7070798601232746)(8.384433343672242,2.)
\psline[linewidth=1.2pt,linecolor=red](6.131948341438671,3.7070798601232746)(6.131948341438671,3.237954573315138)
\psline(7.742057991183408,3.7070798601232746)(7.742057991183407,3.237954573315138)
\psline(8.05176137551848,3.7070798601232746)(8.05176137551848,3.306124117223481)
\psline(8.05176137551848,2.783430863866392)(8.05176137551848,2.)
\parametricplot{1.5707963267948966}{4.71238898038469}{1.*0.1283800922678182*cos(t)+0.*0.1283800922678182*sin(t)+5.940069989396551|0.*0.1283800922678182*cos(t)+1.*0.1283800922678182*sin(t)+3.0447774905449365}
\psline(5.940069989396551,3.1731575828127547)(7.908908732089605,3.1731575828127547)
\psline(2.,3.237954573315138)(2.,2.8516004077747352)
\psline[linewidth=1.2pt,linecolor=red](6.131948341438671,3.0611484297627505)(6.131948341438671,2.)
\psline(7.742057991183407,3.0611484297627505)(7.742057991183408,2.)
\psline(5.940069989396551,2.9163973982771183)(6.05961434464451,2.9163973982771187)
\psline(7.814391987977569,2.9163973982771196)(7.908908732089605,2.9163973982771183)
\parametricplot{-1.7485207621962244}{0.10606967380701743}{1.*0.12197428168953435*cos(t)+0.*0.12197428168953435*sin(t)+7.930472603400021|0.*0.12197428168953435*cos(t)+1.*0.12197428168953435*sin(t)+3.2932105913039447}
\parametricplot{-0.10606967380700993}{1.7485207621962229}{1.*0.12197428168953466*cos(t)+0.*0.12197428168953466*sin(t)+7.930472603400021|0.*0.12197428168953466*cos(t)+1.*0.12197428168953466*sin(t)+2.7963443897859275}
\psline[linewidth=1.2pt,linecolor=red](11.48811247356665,3.7070798601232746)(11.488112473566648,2.)
\psline(13.481200875802337,3.7070798601232746)(13.481200875802335,2.)
\psline(14.041825493191979,3.7070798601232746)(14.041825493191979,2.)
\parametricplot{1.5707963267948966}{4.71238898038469}{1.*0.1283800922678182*cos(t)+0.*0.1283800922678182*sin(t)+11.704763087705427|0.*0.1283800922678182*cos(t)+1.*0.1283800922678182*sin(t)+3.044777490544936}
\parametricplot{-1.7157084569597787}{0.13888197904355062}{1.*0.057446455485935466*cos(t)+0.*0.057446455485935466*sin(t)+13.1852815502023|0.*0.057446455485935466*cos(t)+1.*0.057446455485935466*sin(t)+3.230001918938484}
\psline(11.704763087705427,2.916397398277118)(11.796911288842326,2.916397398277117)
\psline(11.704763087705427,3.1731575828127543)(13.17698596714129,3.17315758281275)
\psline[linewidth=1.2pt,linecolor=red](11.87995636565693,3.7070798601232746)(11.879956365656932,3.237954573315138)
\psline[linewidth=1.2pt,linecolor=red](11.87995636565693,3.0611484297627505)(11.87995636565693,2.)
\psline(13.242174876140087,3.7070798601232746)(13.242174876140085,3.2379545733151263)
\parametricplot{-0.13888197904361554}{1.7157084569597087}{1.*0.05744645548593677*cos(t)+0.*0.05744645548593677*sin(t)+13.185281550202296|0.*0.05744645548593677*cos(t)+1.*0.05744645548593677*sin(t)+2.8595530621513756}
\psline(13.242174876140083,2.8516004077747295)(13.242174876140085,2.)
\small
\rput[tl](6,4.05){$I$}
\rput[tl](6.275,4.05){$II$}
\normalsize
\rput[tl](6.7,4.05){$1$}
\rput[tl](7.05,4.05){$2$}
\rput[tl](8,4.05){$i$}
\rput[tl](8.9,4){$n$}
\small
\rput[tl](11.4,4.05){$I$}
\rput[tl](11.675,4.05){$II$}
\normalsize
\rput[tl](12.055,4.054){$1$}
\rput[tl](12.35,4.05){$2$}
\rput[tl](13.9,4){$n$}
\rput[tl](13.15,4.05){$i$}
\rput[tl](7.410988725817689,1.9){$\mathcal{T}_i$}
\rput[tl](12.640815465251714,1.85){$\tau_i$}
\psline(12.430563810647358,3.7070798601232746)(12.430563810647358,3.707079860123274)
\psline(6.204282338232833,2.9163973982771187)(6.4100009005475105,2.9163973982771183)
\psline(6.554668894135833,2.9163973982771183)(6.718674771223962,2.9163973982771183)
\psline(6.863342764812285,2.9163973982771183)(7.077403864172269,2.9163973982771183)
\psline(7.222071857760592,2.9163973982771183)(7.669723994389246,2.9163973982771196)
\psline[linewidth=1.2pt,linecolor=red](6.482334897341671,3.7070798601232746)(6.48233489734167,3.2379545733151374)
\psline(6.791008768018123,3.7070798601232746)(6.791008768018122,3.237954573315137)
\psline(7.149737860966432,3.7070798601232746)(7.149737860966431,3.2379545733151365)
\psline[linewidth=1.2pt,linecolor=red](6.48233489734167,3.0611484297627505)(6.482334897341671,2.)
\psline(6.791008768018122,3.06114842976275)(6.791008768018122,2.)
\psline(7.149737860966431,3.0611484297627496)(7.1497378609664315,2.)
\psline(11.963001442471535,2.916397398277115)(12.08055754838285,2.9163973982771143)
\psline(12.246647702012059,2.9163973982771125)(12.347518733832752,2.9163973982771143)
\psline(12.513608887461961,2.9163973982771125)(12.885336381530541,2.9163973982771116)
\psline(13.05142653515975,2.91639739827711)(13.17698596714129,2.9163973982771134)
\psline(12.163602625197454,3.7070798601232746)(12.163602625197454,3.237954573315137)
\psline(12.430563810647358,3.7070798601232746)(12.430563810647357,3.2379545733151374)
\psline(12.163602625197452,3.0611484297627496)(12.163602625197452,2.)
\psline(12.430563810647357,3.06114842976275)(12.430563810647358,2.)
\psline(12.968381458345146,3.7070798601232746)(12.968381458345146,3.237954573315136)
\psline(12.968381458345146,3.061148429762749)(12.968381458345144,2.)
\begin{scriptsize}
\psdots[dotsize=3pt 0,dotstyle=*,linecolor=green](2.,2.)
\psdots[dotsize=3pt 0,dotstyle=*,linecolor=green](2.,3.7070798601232746)
\psdots[dotsize=3pt 0,dotstyle=*,linecolor=green](16.88064511893745,3.7070798601232746)
\psdots[dotsize=3pt 0,dotstyle=*,linecolor=darkgray](16.88064511893745,2.)
\psdots[dotsize=3pt 0,dotstyle=*](2.,3.237954573315138)
\psdots[dotsize=3pt 0,dotstyle=*,linecolor=darkgray](16.880645118937455,3.237954573315138)
\psdots[dotsize=3pt 0,dotstyle=*](2.,2.8516004077747352)
\psdots[dotsize=3pt 0,dotstyle=*,linecolor=darkgray](16.880645118937455,2.8516004077747352)
\psdots[dotsize=1pt 0,dotstyle=*](8.69727848287135,2.853539930061637)
\psdots[dotsize=1pt 0,dotstyle=*](8.484468855248481,2.853539930061637)
\psdots[dotsize=1pt 0,dotstyle=*](8.91008811049422,2.853539930061637)
\psdots[dotsize=3pt 0,dotstyle=*](2.,3.0611484297627505)
\psdots[dotsize=1pt 0,dotstyle=*](13.761513184497158,2.8516004077747295)
\psdots[dotsize=1pt 0,dotstyle=*](13.606180886048433,2.8516004077747263)
\psdots[dotsize=1pt 0,dotstyle=*](13.916845482945883,2.8516004077747326)
\psdots[dotsize=1pt 0,dotstyle=*](7.445897926074919,2.566214723628799)
\psdots[dotsize=1pt 0,dotstyle=*](7.314330084516314,2.566214723628799)
\psdots[dotsize=1pt 0,dotstyle=*,](7.577465767633525,2.566214723628799)
\psdots[dotsize=1pt 0,dotstyle=*](12.699472634496253,2.5662147236288004)
\psdots[dotsize=1pt 0,dotstyle=*](12.538227603606185,2.566214723628805)
\psdots[dotsize=1pt 0,dotstyle=*](12.860717665386321,2.566214723628796)
\end{scriptsize}
\end{pspicture*}
\caption{The looping elements $\mathcal{T}_i,\tau_i$.}
\label{figure_3}
\end{figure}
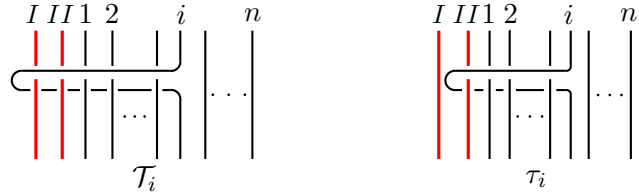

\begin{defn} \rm
The  {\it mixed Hecke algebra on two fixed strands}, denoted $\mathrm{H}_{2,n}(q)$, is the unital, associative algebra defined as the quotient of the group algebra ${\mathbb Z}\, [q^{\pm 1}] \, B_{2,n}$ over the quadratic relations of the classical Iwahori--Hecke algebra for the braiding generators: 
\[ 
\mathrm{H}_{2,n}(q) := \frac{{\mathbb Z}\, [q^{\pm 1}] \, B_{2,n}}{\left<{\sigma_i}^2 - (q-1) \, \sigma_i -
q \cdot 1, \  i=1,2,\ldots,n-1\right>},
\]
\noindent where $q$ is a variable and $1$ denotes the algebra's unit.
\end{defn}

In general we are going to use the same notation for the elements of $B_{2,n}$ when considered as elements of  $\mathrm{H}_{2,n}(q)$, except for $\sigma_i$ which we will be denoting as $g_i$. The algebra $\mathrm{H}_{2,n}(q)$ has an equivalent  presentation with generators $\tau, \mathcal{T},g_1, \ldots , g_{n-1}$ and relations:

\begin{equation*}
\begin{array}{ccrclcll} \label{main_relations}
& & g_k g_{k+1} g_k & = & g_{k+1} g_k g_{k+1} & & \mbox{for} & 1 \leq k \leq n-1 \\
& & g_k g_j & = & g_j g_k & &\mbox{for}&   |k-j|>1\\
& & \mathcal{T}\, g_k & =  &  g_k \,\mathcal{T} & &  \mbox{for} &   k \geq 2 \\
& & \tau\, g_k & = & g_k\, \tau  & & \mbox{for} &   k \geq 2 \\
& &   \mathcal{T}\, g_1 \,\mathcal{T}\, g_1 & = & g_1 \,\mathcal{T}\, g_1\, \mathcal{T}  &&& \\
& &  \tau\, g_1 \,\tau\, g_1 & =& g_1 \,\tau\, g_1\, \tau  &&& \\
& &  \tau (g_1 \mathcal{T} {g_1}) & =& (g_1 \mathcal{T} {g_1}) \tau  &&& \\
& &  g_i^2 & = & (q-1) \, g_i + q \cdot 1  & & \mbox{for} &  $ i=1,2,\ldots,n-1$
\end{array}
\end{equation*}
Borrowing the terminology from the braid group level, we call the elements $\tau,\mathcal{T}$ of the algebra \textit{looping generators} and $g_i$  \textit{braiding generators} of the algebra. Similarly, considering them as elements of  $\mathrm{H}_{2,n}(q)$, we call  $\mathcal{T}_i,\tau_i$ and their inverses \textit{looping elements} or just \textit{loopings}. Also, we say that $i$  is the \textit{index} of the elements $\mathcal{T}_i^{\pm 1}, \tau_i^{\pm 1}$ and $g_i^{\pm 1}$. Since the image of $\sigma_i$ in the algebra is $g_i$, we have:
\begin{equation} \nonumber
  \mathcal{T}_i = g_{i-1} \ldots g_1\mathcal{T} g_1 \ldots g_{i-1} \quad \mbox{and} \quad \tau_i = g_{i-1} \ldots g_1\tau g_1 \ldots g_{i-1}.
	\end{equation}

\section{The spanning set $\Lambda_n$ of $\mathrm{H}_{2,n}(q)$ as a $\mathbb{Z}[q^{\pm 1}]$-module}

Our aim is to put any element $w$ of $\mathrm{H}_{2,n}(q)$ in a form suitable for constructing a Markov trace in the sequence of the algebras $\mathrm{H}_{2,n}(q),\ n \in \mathbb{N}$, and, eventually, for constructing Homflypt-type invariants for the oriented links in the 3-manifolds whose structures are encoded by these algebras. Previous work done with $\mathrm{H}_{1,n}(q)$ (cf. \cite{La1} and references therein), indicates that we should  express each $w$ via  looping elements like the $\tau_i^{\pm}$'s and $\mathcal{T}_i^{\pm}$'s, as a polynomial of the looping elements with each term being a monomial with ordered indices for the loopings, increasing from left to right, and all these followed by some tail products of $g_i$'s. Here we wish to write any $w$ in $\mathrm{H}_{2,n}(q)$ as a $\mathbb{Z}[q^\pm]$-linear sum of elements of the set   
$\Lambda_n$ whose elements are the products $ \Pi_1\Pi_2\cdots\Pi_n G$ with each $\Pi_i$ a finite product of elements of only the loopings $\{\mathcal{T}_i,\tau_i,\mathcal{T}_i^{-1},\tau_i^{-1}\}$ and 
 $ G$ a finite product of braiding generators. In other words, we wish to show that $\Lambda_n$ is a spanning set for the additive structure of $\mathrm{H}_{2,n}(q)$ as a $\mathbb{Z}[q^{\pm 1}]$-module.

We already know that we can achieve the above goal for certain subsets of $\mathrm{H}_{2,n}(q)$. For example whenever $w$ is a product of only the  $g_i$'s then it automatically gets the desired form. By the way, each such $w$  actually belongs to the  Iwahori--Hecke algebra, $\mathrm{H}_n(q)$, of type $\mathrm{A}$ and as such it is subjected to the canonical form given by V.F.R. Jones \cite{J}. Also, whenever $w$ is a product of only  $\tau_i$'s and $g_i$'s,  (thus containing no $\mathcal{T}_i$'s)or the analogous situation of only containing $\mathcal{T}_i$'s and $g_i$'s, it actually belongs to $\mathrm{H}_{1,n}(q)$ (which is the `generalized' Hecke algebra  of type $\mathrm{B}$ \cite{La1}), and therefore it is subjected to the canonical  form  given in \cite{La1}. Such a $w$ is written as a finite $\mathbb{Z}[q^{\pm 1}]$-linear combination of products of $\tau_i$'s and $g_i$'s with the $\tau_i$'s appearing first, and moreover with the indices of the $\tau_i$'s in increasing order from left to right.

To achieve our goal for all elements in $\mathrm{H}_{2,n}(q)$ it is of course enough to  focus on images in the algebra of braids $w$ in $B_{2,n}$. So, we first write $w$ as a product in the generators $\mathcal{T}_i, \tau_i, g_j$ of $B_{2,n}$. Then, in the algebra level, we quite easily push all $g_i$'s (i.e. the images of the $\sigma_i$'s) at the end of this product as described by Lemma \ref{lemma_2} below, using  the quadratic relations of the algebra along with braid isotopies at the braid level. So it is now enough to prove that all words $w$ in $\mathrm{H}_{2,n}(q)$ which are products of loopings can be written in the desired way. This is done in part A of Theorem \ref{theorem_2} below. The proof is a bit technical since we need to pay extra attention to the indices of the loopings involved, a necessary evil as the proof goes by a kind of induction in pairs of such indices.

More disturbing than these technicalities is the fact that certain recursion phenomena occur whenever we try to achieve the left to right increasing order of the looping indices: in the process of pushing loopings of bigger indices to the right of others with smaller indices (as described in Lemma \ref{lemma_3}), some new $g_i$'s might be created, and pushing them anew to the end might increase the indices of the loopings from which it passes, leaving quite open the question of whether the indices of the loopings can indeed be ordered. 

Examples indicate that this kind of reordering the indices always ends the way we wish, except in a certain initial arrangement of the indices, in which case our original $w$ is expressed in terms of itself. Fortunately, we can always solve the equation that arises.  We deal with these recursion phenomena in Lemma \ref{lemma_4}. Figure \ref{figure_last} exhibits schematically a typical example of this recursion phenomenon: the equalities are among elements in the algebra $\mathrm{H}_{2,n}(q)$; the first and last equality can be seen via isotopies in the braid level, while the second one via an application of the quadratic relation in the circled crossing. Thus we obtain $w=q^{-1}\mathcal{T}_1w'+Awg_1$, from which $w(1-Ag_1)=q^{-1}\mathcal{T}_1w'$ where $w'$ is settled by Lemmata \ref{lemma_2}, \ref{lemma_3}. But $1-Ag_1=q^{-1}g_1^2$, thus $w=\mathcal{T}_1w'g_1^{-2}$ is settled in the way described by the Theorem.

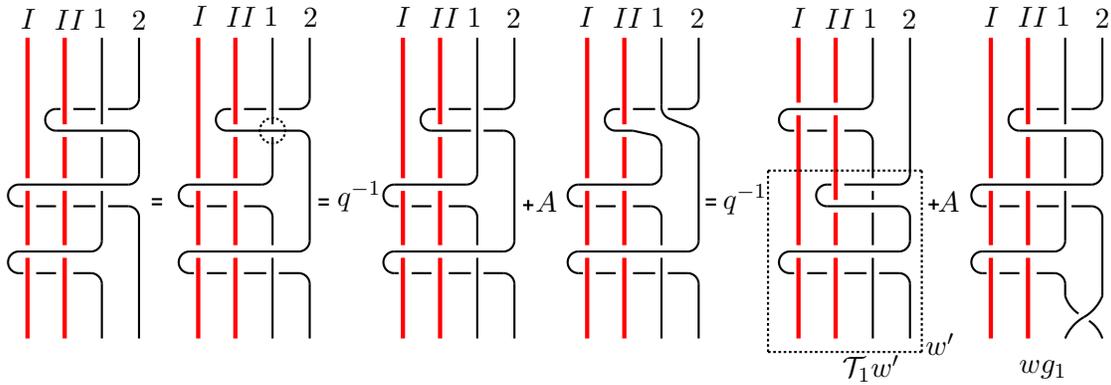
\begin{figure}[h]
\centering
\psset{xunit=1.0cm,yunit=1.0cm,algebraic=true,dimen=middle,dotstyle=o,dotsize=5pt 0,linewidth=0.8pt,arrowsize=3pt 2,arrowinset=0.25}
\begin{pspicture*}(2.4517230957061673,0.3338378389568005)(17.309773694642452,5.5)
\rput[tl](1.4611863891104149,11.589936777545503){$w_1$}
\rput[tl](1.5152156640156378,9.527819451996052){$w_2$}
\psline(3.4644534836815195,4.0528529282665104)(3.758806179386551,4.0528529282665104)
\psline(3.7588061793865517,2.759477017658749)(3.4644534836815195,2.759477017658749)
\psline(3.2338987811489743,2.759477017658749)(2.970433931066568,2.759477017658749)
\psline(3.2338987811489743,1.8691287040186242)(2.970433931066568,1.8691287040186242)
\psline(6.02789673990659,4.052852928266512)(5.733544044201559,4.052852928266512)
\psline(5.502989341669013,2.7594770176587504)(5.239524491586608,2.7594770176587504)
\psline(6.02789673990659,1.8691287040186249)(5.733544044201559,1.8691287040186246)
\psline(5.5029893416690125,1.8691287040186244)(5.239524491586609,1.8691287040186244)
\psline(8.749295306663841,4.05285292826651)(8.454942610958806,4.052852928266509)
\psline(8.749295306663841,3.764696795438639)(8.224387908426262,3.76469679543864)
\psline(8.224387908426262,2.759477017658752)(7.960923058343854,2.7594770176587518)
\psline(8.74929530666384,1.869128704018627)(8.454942610958806,1.8691287040186266)
\psline(8.224387908426262,1.869128704018626)(7.960923058343854,1.8691287040186257)
\psline(13.483237332534623,1.8691287040186297)(13.219772482452214,1.8691287040186297)
\psline(16.56553040961941,4.052852928266508)(16.27117771391438,4.052852928266508)
\psline(16.56553040961941,2.7594770176587486)(16.27117771391438,2.7594770176587486)
\psline(16.040623011381836,1.8691287040186249)(15.777158161299432,1.8691287040186249)
\parametricplot{1.5707963267948966}{4.71238898038469}{1.*0.14407806641393517*cos(t)+0.*0.14407806641393517*sin(t)+2.7398792285340225|0.*0.14407806641393517*cos(t)+1.*0.14407806641393517*sin(t)+2.9035550840726843}
\parametricplot{1.5707963267948966}{4.71238898038469}{1.*0.14407806641393517*cos(t)+0.*0.14407806641393517*sin(t)+2.7398792285340225|0.*0.14407806641393517*cos(t)+1.*0.14407806641393517*sin(t)+2.01320677043256}
\parametricplot{1.5707963267948966}{4.71238898038469}{1.*0.14407806641393517*cos(t)+0.*0.14407806641393517*sin(t)+5.0089697890540625|0.*0.14407806641393517*cos(t)+1.*0.14407806641393517*sin(t)+2.01320677043256}
\parametricplot{1.5707963267948966}{4.71238898038469}{1.*0.14407806641393517*cos(t)+0.*0.14407806641393517*sin(t)+5.0089697890540625|0.*0.14407806641393517*cos(t)+1.*0.14407806641393517*sin(t)+2.9035550840726843}
\parametricplot{1.5707963267948966}{4.71238898038469}{1.*0.14407806641393517*cos(t)+0.*0.14407806641393517*sin(t)+3.233898781148974|0.*0.14407806641393517*cos(t)+1.*0.14407806641393517*sin(t)+3.9087748618525753}
\parametricplot{1.5707963267948966}{4.71238898038469}{1.*0.14407806641393517*cos(t)+0.*0.14407806641393517*sin(t)+5.502989341669013|0.*0.14407806641393517*cos(t)+1.*0.14407806641393517*sin(t)+3.9087748618525753}
\parametricplot{1.5707963267948966}{4.71238898038469}{1.*0.14407806641393517*cos(t)+0.*0.14407806641393517*sin(t)+7.73036835581131|0.*0.14407806641393517*cos(t)+1.*0.14407806641393517*sin(t)+2.01320677043256}
\parametricplot{1.5707963267948966}{4.71238898038469}{1.*0.14407806641393517*cos(t)+0.*0.14407806641393517*sin(t)+7.73036835581131|0.*0.14407806641393517*cos(t)+1.*0.14407806641393517*sin(t)+2.9035550840726843}
\parametricplot{1.5707963267948966}{4.71238898038469}{1.*0.14407806641393517*cos(t)+0.*0.14407806641393517*sin(t)+8.224387908426262|0.*0.14407806641393517*cos(t)+1.*0.14407806641393517*sin(t)+3.9087748618525753}
\parametricplot{1.5707963267949028}{4.712388980384696}{1.*0.14407806641393517*cos(t)+0.*0.14407806641393517*sin(t)+10.179695484848082|0.*0.14407806641393517*cos(t)+1.*0.14407806641393517*sin(t)+2.0132067704325625}
\parametricplot{1.5707963267948966}{4.71238898038469}{1.*0.14407806641393517*cos(t)+0.*0.14407806641393517*sin(t)+10.17969548484808|0.*0.14407806641393517*cos(t)+1.*0.14407806641393517*sin(t)+2.903555084072687}
\parametricplot{1.5707963267948966}{4.71238898038469}{1.*0.14407806641393517*cos(t)+0.*0.14407806641393517*sin(t)+10.673715037463031|0.*0.14407806641393517*cos(t)+1.*0.14407806641393517*sin(t)+3.908774861852578}
\parametricplot{1.5707963267949028}{4.712388980384696}{1.*0.14407806641393517*cos(t)+0.*0.14407806641393517*sin(t)+12.989217779919674|0.*0.14407806641393517*cos(t)+1.*0.14407806641393517*sin(t)+2.013206770432564}
\parametricplot{1.5707963267948966}{4.71238898038469}{1.*0.14407806641393517*cos(t)+0.*0.14407806641393517*sin(t)+12.989217779919672|0.*0.14407806641393517*cos(t)+1.*0.14407806641393517*sin(t)+3.908774861852579}
\parametricplot{1.5707963267948966}{4.71238898038469}{1.*0.14407806641393517*cos(t)+0.*0.14407806641393517*sin(t)+13.483237332534621|0.*0.14407806641393517*cos(t)+1.*0.14407806641393517*sin(t)+2.9035550840726883}
\parametricplot{1.5707963267948966}{4.71238898038469}{1.*0.14407806641393517*cos(t)+0.*0.14407806641393517*sin(t)+16.040623011381836|0.*0.14407806641393517*cos(t)+1.*0.14407806641393517*sin(t)+3.9087748618525753}
\parametricplot{1.5707963267948966}{4.71238898038469}{1.*0.14407806641393517*cos(t)+0.*0.14407806641393517*sin(t)+15.546603458766885|0.*0.14407806641393517*cos(t)+1.*0.14407806641393517*sin(t)+2.9035550840726843}
\parametricplot{1.5707963267948966}{4.71238898038469}{1.*0.14407806641393517*cos(t)+0.*0.14407806641393517*sin(t)+15.546603458766885|0.*0.14407806641393517*cos(t)+1.*0.14407806641393517*sin(t)+2.01320677043256}
\psline(11.198622435700608,4.0528529282665104)(10.904269739995577,4.05285292826651)
\psline(10.410250187380626,2.7594770176587518)(10.673715037463031,2.759477017658751)
\psline(11.198622435700608,1.8691287040186282)(10.904269739995575,1.8691287040186282)
\psline(10.673715037463031,1.8691287040186277)(10.410250187380626,1.8691287040186277)
\psline(13.713792035067163,1.86912870401863)(14.008144730772196,1.8691287040186304)
\psline(3.233898781148974,4.0528529282665104)(3.29153745678211,4.0528529282665104)
\psline(2.7398792285340225,2.759477017658749)(2.797517904167159,2.759477017658749)
\psline(2.7398792285340225,1.8691287040186246)(2.797517904167159,1.8691287040186244)
\psline(5.502989341669013,4.0528529282665104)(5.56062801730215,4.052852928266511)
\psline(5.0089697890540625,2.759477017658749)(5.066608464687199,2.7594770176587495)
\psline(5.0089697890540625,1.8691287040186246)(5.0666084646871985,1.8691287040186246)
\psline(7.73036835581131,2.759477017658749)(7.788007031444446,2.7594770176587504)
\psline(7.73036835581131,1.8691287040186246)(7.788007031444446,1.869128704018625)
\psline(8.224387908426262,4.0528529282665104)(8.282026584059398,4.05285292826651)
\psline(10.673715037463031,4.052852928266513)(10.731353713096167,4.052852928266511)
\psline(10.17969548484808,2.7594770176587518)(10.237334160481216,2.7594770176587518)
\psline(10.179695484848082,1.8691287040186273)(10.237334160481218,1.8691287040186275)
\psline(12.989217779919674,1.8691287040186286)(13.046856455552808,1.869128704018629)
\psline(13.483237332534621,3.0476331504866234)(13.540876008167757,3.0476331504866265)
\psline(16.040623011381836,4.0528529282665104)(16.098261687014972,4.052852928266509)
\psline(15.546603458766885,2.759477017658749)(15.604242134400021,2.7594770176587486)
\psline(15.546603458766885,1.8691287040186246)(15.604242134400023,1.8691287040186246)
\psline(15.77715816129943,2.7594770176587486)(16.040623011381836,2.7594770176587486)
\psline[linewidth=1.6pt,linecolor=red](2.8551565798002954,5.)(2.8551565798002954,3.132022656130266)
\psline[linewidth=1.6pt,linecolor=red](2.8551565798002954,2.9632436448429718)(2.8551565798002954,2.2416743424901417)
\psline[linewidth=1.6pt,linecolor=red](2.8551565798002954,2.0728953312028477)(2.8551565798002954,1.)
\psline[linewidth=1.6pt,linecolor=red](3.3491761324152467,5.)(3.3491761324152463,3.8490863010822878)
\psline[linewidth=1.6pt,linecolor=red](3.3491761324152463,3.680307289794993)(3.349176132415246,3.132022656130266)
\psline[linewidth=1.6pt,linecolor=red](3.349176132415246,2.9632436448429718)(3.3491761324152463,2.2416743424901417)
\psline[linewidth=1.6pt,linecolor=red](3.3491761324152463,2.0728953312028477)(3.3491761324152467,1.)
\psline(3.843195685030198,5.)(3.843195685030198,3.8490863010822873)
\psline(3.843195685030198,3.680307289794993)(3.8431956850301985,3.132022656130266)
\psline[linewidth=1.6pt,linecolor=red](5.124247140320335,5.)(5.124247140320335,3.1320226561302666)
\psline[linewidth=1.6pt,linecolor=red](5.124247140320335,2.963243644842972)(5.124247140320335,2.2416743424901417)
\psline[linewidth=1.6pt,linecolor=red](5.124247140320335,2.0728953312028473)(5.124247140320335,1.)
\psline[linewidth=1.6pt,linecolor=red](5.618266692935286,5.)(5.618266692935286,3.84908630108229)
\psline[linewidth=1.6pt,linecolor=red](5.618266692935286,3.680307289794995)(5.618266692935286,3.1320226561302666)
\psline[linewidth=1.6pt,linecolor=red](5.618266692935286,2.963243644842972)(5.618266692935287,2.2416743424901417)
\psline[linewidth=1.6pt,linecolor=red](5.618266692935287,2.072895331202848)(5.618266692935286,1.)
\psline(6.112286245550237,5.)(6.112286245550237,3.8490863010822904)
\psline(6.112286245550237,2.0728953312028473)(6.112286245550237,1.)
\psline[linewidth=1.6pt,linecolor=red](7.8456457070775825,5.)(7.845645707077582,3.1320226561302684)
\psline[linewidth=1.6pt,linecolor=red](7.845645707077582,2.963243644842974)(7.8456457070775825,2.241674342490142)
\psline[linewidth=1.6pt,linecolor=red](7.8456457070775825,2.0728953312028477)(7.8456457070775825,1.)
\psline[linewidth=1.6pt,linecolor=red](8.339665259692534,5.)(8.339665259692534,3.849086301082285)
\psline[linewidth=1.6pt,linecolor=red](8.339665259692534,3.68030728979499)(8.339665259692534,3.1320226561302698)
\psline[linewidth=1.6pt,linecolor=red](8.339665259692534,2.9632436448429753)(8.339665259692534,2.241674342490143)
\psline[linewidth=1.6pt,linecolor=red](8.339665259692534,2.072895331202848)(8.339665259692534,1.)
\psline(8.833684812307485,2.072895331202849)(8.833684812307487,1.)
\psline[linewidth=1.6pt,linecolor=red](10.294972836114352,5.)(10.294972836114352,3.1320226561302698)
\psline[linewidth=1.6pt,linecolor=red](10.294972836114352,2.9632436448429753)(10.294972836114354,2.241674342490144)
\psline[linewidth=1.6pt,linecolor=red](10.294972836114354,2.072895331202849)(10.294972836114354,1.)
\psline[linewidth=1.6pt,linecolor=red](10.788992388729303,5.)(10.788992388729303,3.8490863010822913)
\psline[linewidth=1.6pt,linecolor=red](10.788992388729303,3.680307289794996)(10.788992388729303,3.1320226561302684)
\psline[linewidth=1.6pt,linecolor=red](10.788992388729303,2.9632436448429735)(10.788992388729303,2.2416743424901444)
\psline[linewidth=1.6pt,linecolor=red](10.788992388729303,2.07289533120285)(10.788992388729303,1.)
\psline(10.673715037463031,3.764696795438643)(10.904269739995577,3.7646967954386437)
\psline(11.283011941344252,2.07289533120285)(11.283011941344254,1.)
\psline(13.483237332534621,3.764696795438642)(13.219772482452214,3.764696795438643)
\psline(12.989217779919672,3.764696795438644)(13.046856455552806,3.7646967954386437)
\parametricplot{4.7779881522157375}{6.2175861353485455}{1.*0.14863080085190597*cos(t)+0.*0.14863080085190597*sin(t)+4.188904119975646|0.*0.14863080085190597*cos(t)+1.*0.14863080085190597*sin(t)+4.201164045936013}
\parametricplot{4.777988152215716}{6.2175861353485}{1.*0.14863080085189953*cos(t)+0.*0.14863080085189953*sin(t)+4.188904119975654|0.*0.14863080085189953*cos(t)+1.*0.14863080085189953*sin(t)+3.1959442681561154}
\parametricplot{4.7779881522156975}{6.217586135348524}{1.*0.1486308008518878*cos(t)+0.*0.1486308008518878*sin(t)+3.694884567360716|0.*0.1486308008518878*cos(t)+1.*0.1486308008518878*sin(t)+2.305595954515979}
\parametricplot{4.777988152215805}{6.217586135348626}{1.*0.14863080085188796*cos(t)+0.*0.14863080085188796*sin(t)+6.4579946804957045|0.*0.14863080085188796*cos(t)+1.*0.14863080085188796*sin(t)+4.201164045935993}
\parametricplot{4.777988152215736}{6.2175861353485296}{1.*0.14863080085188327*cos(t)+0.*0.14863080085188327*sin(t)+6.457994680495705|0.*0.14863080085188327*cos(t)+1.*0.14863080085188327*sin(t)+2.305595954515975}
\parametricplot{4.7779881522157375}{6.217586135348614}{1.*0.14863080085187927*cos(t)+0.*0.14863080085187927*sin(t)+5.96397512788076|0.*0.14863080085187927*cos(t)+1.*0.14863080085187927*sin(t)+3.1959442681560906}
\parametricplot{4.77798815221569}{6.217586135348496}{1.*0.14863080085191707*cos(t)+0.*0.14863080085191707*sin(t)+9.17939324725293|0.*0.14863080085191707*cos(t)+1.*0.14863080085191707*sin(t)+4.20116404593603}
\parametricplot{4.7779881522157845}{6.2175861353485935}{1.*0.148630800851922*cos(t)+0.*0.148630800851922*sin(t)+9.179393247252918|0.*0.148630800851922*cos(t)+1.*0.148630800851922*sin(t)+2.305595954516014}
\parametricplot{4.777988152215903}{6.217586135348628}{1.*0.14863080085192407*cos(t)+0.*0.14863080085192407*sin(t)+8.68537369463796|0.*0.14863080085192407*cos(t)+1.*0.14863080085192407*sin(t)+3.1959442681561385}
\parametricplot{4.777988152215911}{6.217586135348725}{1.*0.14863080085182578*cos(t)+0.*0.14863080085182578*sin(t)+11.628720376289783|0.*0.14863080085182578*cos(t)+1.*0.14863080085182578*sin(t)+4.201164045935941}
\parametricplot{4.777988152215715}{6.217586135348535}{1.*0.14863080085186126*cos(t)+0.*0.14863080085186126*sin(t)+11.628720376289753|0.*0.14863080085186126*cos(t)+1.*0.14863080085186126*sin(t)+2.3055959545159554}
\parametricplot{4.777988152215592}{6.217586135348347}{1.*0.148630800851925*cos(t)+0.*0.148630800851925*sin(t)+11.13470082367473|0.*0.148630800851925*cos(t)+1.*0.148630800851925*sin(t)+3.1959442681561487}
\parametricplot{4.777988152215528}{6.217586135348406}{1.*0.1486308008518256*cos(t)+0.*0.1486308008518256*sin(t)+13.944223118746422|0.*0.1486308008518256*cos(t)+1.*0.1486308008518256*sin(t)+4.2011640459359345}
\parametricplot{4.777988152215656}{6.217586135348442}{1.*0.14863080085196792*cos(t)+0.*0.14863080085196792*sin(t)+14.43824267136124|0.*0.14863080085196792*cos(t)+1.*0.14863080085196792*sin(t)+3.195944268156183}
\parametricplot{4.7779881522157375}{6.217586135348544}{1.*0.14863080085186461*cos(t)+0.*0.14863080085186461*sin(t)+14.438242671361337|0.*0.14863080085186461*cos(t)+1.*0.14863080085186461*sin(t)+2.30559595451596}
\parametricplot{4.777988152215784}{6.2175861353485695}{1.*0.14863080085186814*cos(t)+0.*0.14863080085186814*sin(t)+16.501608797593597|0.*0.14863080085186814*cos(t)+1.*0.14863080085186814*sin(t)+2.305595954515955}
\parametricplot{4.777988152215251}{6.217586135348097}{1.*0.14863080085188163*cos(t)+0.*0.14863080085188163*sin(t)+16.995628350208534|0.*0.14863080085188163*cos(t)+1.*0.14863080085188163*sin(t)+4.201164045936013}
\parametricplot{4.777988152215519}{6.217586135348457}{1.*0.14863080085181882*cos(t)+0.*0.14863080085181882*sin(t)+16.995628350208587|0.*0.14863080085181882*cos(t)+1.*0.14863080085181882*sin(t)+3.195944268156042}
\psline[linewidth=1.6pt,linecolor=red](13.104495131185942,5.)(13.104495131185942,4.137242433910164)
\psline[linewidth=1.6pt,linecolor=red](13.104495131185942,3.9684634226228703)(13.104495131185942,2.2416743424901457)
\psline[linewidth=1.6pt,linecolor=red](13.104495131185942,2.0728953312028513)(13.104495131185942,1.)
\psline[linewidth=1.6pt,linecolor=red](13.598514683800893,5.)(13.598514683800893,4.1372424339101626)
\psline[linewidth=1.6pt,linecolor=red](13.598514683800893,3.9684634226228686)(13.598514683800893,2.843866523302396)
\psline[linewidth=1.6pt,linecolor=red](13.598514683800893,2.6750875120151014)(13.598514683800893,2.241674342490146)
\psline[linewidth=1.6pt,linecolor=red](13.598514683800893,2.072895331202852)(13.598514683800895,1.)
\psline(14.092534236415844,2.9632436448429824)(14.092534236415844,2.8438665233023963)
\psline(14.092534236415844,2.675087512015102)(14.092534236415844,2.241674342490146)
\psline(14.092534236415844,2.0728953312028517)(14.092534236415844,1.)
\psline[linewidth=1.6pt,linecolor=red](15.661880810033157,5.)(15.661880810033157,3.1320226561302693)
\psline[linewidth=1.6pt,linecolor=red](15.661880810033157,2.963243644842975)(15.66188081003316,2.241674342490143)
\psline[linewidth=1.6pt,linecolor=red](15.66188081003316,2.0728953312028486)(15.661880810033159,1.)
\psline[linewidth=1.6pt,linecolor=red](16.155900362648108,5.)(16.155900362648108,3.8490863010822873)
\psline[linewidth=1.6pt,linecolor=red](16.155900362648108,3.680307289794992)(16.155900362648108,3.1320226561302693)
\psline[linewidth=1.6pt,linecolor=red](16.155900362648104,2.963243644842975)(16.155900362648108,2.241674342490143)
\psline[linewidth=1.6pt,linecolor=red](16.155900362648108,2.0728953312028486)(16.15590036264811,1.)
\psline(16.64991991526306,5.)(16.64991991526306,3.8490863010822873)
\psline(16.64991991526306,3.680307289794993)(16.64991991526306,3.1320226561302693)
\psline(16.64991991526306,2.963243644842975)(16.649919915263062,2.295852888391687)
\parametricplot{0.06559917183099628}{1.5051971549638135}{1.*0.1486308008518939*cos(t)+0.*0.1486308008518939*sin(t)+4.18890411997566|0.*0.1486308008518939*cos(t)+1.*0.1486308008518939*sin(t)+3.6163856777691517}
\parametricplot{0.0655991718310494}{1.505197154963811}{1.*0.14863080085190244*cos(t)+0.*0.14863080085190244*sin(t)+4.188904119975651|0.*0.14863080085190244*cos(t)+1.*0.14863080085190244*sin(t)+2.611165899989254}
\parametricplot{0.06559917183101213}{1.5051971549638516}{1.*0.14863080085188515*cos(t)+0.*0.14863080085188515*sin(t)+6.457994680495705|0.*0.14863080085188515*cos(t)+1.*0.14863080085188515*sin(t)+3.616385677769163}
\parametricplot{0.06559917183109797}{1.5051971549639767}{1.*0.14863080085189756*cos(t)+0.*0.14863080085189756*sin(t)+6.4579946804956965|0.*0.14863080085189756*cos(t)+1.*0.14863080085189756*sin(t)+1.7208175863491317}
\parametricplot{0.06559917183110561}{1.5051971549638536}{1.*0.14863080085195818*cos(t)+0.*0.14863080085195818*sin(t)+9.179393247252882|0.*0.14863080085195818*cos(t)+1.*0.14863080085195818*sin(t)+3.6163856777690873}
\parametricplot{0.06559917183111032}{1.505197154963922}{1.*0.14863080085194755*cos(t)+0.*0.14863080085194755*sin(t)+9.179393247252897|0.*0.14863080085194755*cos(t)+1.*0.14863080085194755*sin(t)+1.7208175863490776}
\parametricplot{0.06559917183074701}{1.5051971549636423}{1.*0.1486308008518888*cos(t)+0.*0.1486308008518888*sin(t)+11.13470082367477|0.*0.1486308008518888*cos(t)+1.*0.1486308008518888*sin(t)+2.61116589998927}
\parametricplot{0.06559917183129826}{1.5051971549638632}{1.*0.1486308008518746*cos(t)+0.*0.1486308008518746*sin(t)+11.628720376289738|0.*0.1486308008518746*cos(t)+1.*0.1486308008518746*sin(t)+1.720817586349158}
\parametricplot{0.06559917183135805}{1.5051971549641525}{1.*0.14863080085190186*cos(t)+0.*0.14863080085190186*sin(t)+13.944223118746345|0.*0.14863080085190186*cos(t)+1.*0.14863080085190186*sin(t)+3.6163856777691437}
\parametricplot{0.06559917183067812}{1.5051971549633947}{1.*0.14863080085186856*cos(t)+0.*0.14863080085186856*sin(t)+14.438242671361328|0.*0.14863080085186856*cos(t)+1.*0.14863080085186856*sin(t)+2.611165899989292}
\parametricplot{0.06559917183036805}{1.5051971549629994}{1.*0.1486308008517605*cos(t)+0.*0.1486308008517605*sin(t)+14.438242671361435|0.*0.1486308008517605*cos(t)+1.*0.1486308008517605*sin(t)+1.720817586349281}
\parametricplot{0.06559917183110144}{1.5051971549641803}{1.*0.14863080085180502*cos(t)+0.*0.14863080085180502*sin(t)+16.995628350208605|0.*0.14863080085180502*cos(t)+1.*0.14863080085180502*sin(t)+3.6163856777692422}
\parametricplot{0.06559917183148169}{1.5051971549644043}{1.*0.14863080085187283*cos(t)+0.*0.14863080085187283*sin(t)+16.995628350208545|0.*0.14863080085187283*cos(t)+1.*0.14863080085187283*sin(t)+2.611165899989277}
\parametricplot{0.06559917183132452}{1.5051971549640022}{1.*0.14863080085183214*cos(t)+0.*0.14863080085183214*sin(t)+16.50160879759363|0.*0.14863080085183214*cos(t)+1.*0.14863080085183214*sin(t)+1.7208175863491952}
\psline(4.337215237645149,5.)(4.33721523764515,4.191420979811738)
\psline(4.1986471860999215,4.0528529282665104)(3.927585190673845,4.0528529282665104)
\psline(4.198647186099939,3.7646967954386414)(3.233898781148974,3.76469679543864)
\psline(4.33721523764515,3.6261287438934184)(4.33721523764515,3.186201202031835)
\psline(4.198647186099926,3.0476331504866194)(2.7398792285340225,3.0476331504866194)
\psline(4.198647186099931,2.7594770176587518)(3.927585190673845,2.759477017658749)
\psline(3.704627633484984,2.1572848368464945)(2.7398792285340225,2.157284836846495)
\psline(3.8431956850301985,2.963243644842972)(3.8431956850301994,2.295852888391703)
\parametricplot{0.06559917183101525}{1.5051971549638605}{1.*0.14863080085187808*cos(t)+0.*0.14863080085187808*sin(t)+3.694884567360723|0.*0.14863080085187808*cos(t)+1.*0.14863080085187808*sin(t)+1.7208175863491488}
\psline(3.7046276334849946,1.8691287040186237)(3.4644534836815195,1.8691287040186242)
\psline(3.8431956850301976,1.7305606524734174)(3.843195685030198,1.)
\psline(4.337215237645149,2.620908966113529)(4.337215237645149,1.)
\parametricplot{0.06559917183089153}{1.505197154963778}{1.*0.14863080085189376*cos(t)+0.*0.14863080085189376*sin(t)+5.963975127880748|0.*0.14863080085189376*cos(t)+1.*0.14863080085189376*sin(t)+2.611165899989261}
\parametricplot{0.06559917183093047}{1.5051971549637544}{1.*0.14863080085192087*cos(t)+0.*0.14863080085192087*sin(t)+8.68537369463797|0.*0.14863080085192087*cos(t)+1.*0.14863080085192087*sin(t)+2.6111658999892375}
\psline(6.467737746619989,4.05285292826651)(6.196675751193884,4.052852928266512)
\psline(6.606305798165191,4.191420979811732)(6.606305798165188,5.)
\psline(5.502989341669013,3.76469679543864)(6.467737746619979,3.764696795438645)
\psline(6.606305798165187,3.6261287438934318)(6.606305798165187,2.2958528883916998)
\psline(5.0089697890540625,3.0476331504866194)(5.973718194005033,3.047633150486615)
\psline(6.112286245550237,3.680307289794995)(6.112286245550237,3.186201202031828)
\psline(5.973718194005033,2.7594770176587513)(5.733544044201559,2.7594770176587504)
\psline(6.11228624555024,2.6209089661135123)(6.112286245550237,2.2416743424901417)
\psline(6.467737746619952,1.8691287040186266)(6.196675751193885,1.8691287040186249)
\psline(6.60630579816519,1.7305606524734138)(6.606305798165189,1.)
\psline(6.467737746619979,2.1572848368464954)(5.0089697890540625,2.157284836846495)
\psline(9.1891363133772,4.052852928266517)(8.918074317951133,4.05285292826651)
\psline(9.327704364922443,4.191420979811748)(9.327704364922436,5.)
\psline(9.18913631337716,3.764696795438641)(8.918074317951133,3.764696795438639)
\psline(9.327704364922436,3.6261287438933745)(9.327704364922436,2.295852888391746)
\psline(7.73036835581131,2.157284836846495)(9.189136313377201,2.1572848368464963)
\psline(8.833684812307485,5.)(8.833684812307483,3.1862012020318753)
\psline(8.69511676076226,3.04763315048662)(7.73036835581131,3.0476331504866194)
\psline(8.69511676076226,2.7594770176587557)(8.454942610958806,2.759477017658752)
\psline(8.833684812307489,2.6209089661134963)(8.833684812307485,2.2416743424901435)
\psline(9.189136313377164,1.8691287040186244)(8.918074317951133,1.869128704018627)
\psline(9.32770436492244,1.7305606524733648)(9.327704364922438,1.)
\psline(11.638463442414078,4.05285292826652)(11.3674014469879,4.0528529282665104)
\psline(11.777031493959207,4.191420979811698)(11.777031493959205,5.)
\psline(11.144443889798985,3.047633150486626)(10.17969548484808,3.047633150486622)
\psline(11.144443889799074,2.75947701765875)(10.904269739995573,2.75947701765875)
\psline(11.283011941344258,2.6209089661134994)(11.283011941344254,2.2416743424901444)
\psline(10.17969548484808,2.1572848368464976)(11.638463442414022,2.1572848368464976)
\psline(11.367401446987898,1.8691287040186282)(11.63846344241401,1.869128704018625)
\psline(11.777031493959207,1.7305606524734682)(11.777031493959207,1.)
\psline(12.989217779919672,4.052852928266514)(13.95396618487066,4.0528529282665104)
\psline(13.953966184870575,3.764696795438645)(13.713792035067165,3.764696795438641)
\psline(14.09253423641584,3.6261287438934646)(14.092534236415844,3.132022656130277)
\psline(13.713792035067165,3.0476331504866296)(14.447985737485507,3.047633150486618)
\psline(14.586553789030795,5.)(14.5865537890308,3.1862012020318895)
\psline(13.483237332534621,2.759477017658753)(14.447985737485668,2.7594770176587518)
\psline(14.586553789030797,2.62090896611351)(14.586553789030797,2.295852888391688)
\psline(14.447985737485826,1.8691287040186286)(14.17692374205949,1.8691287040186304)
\psline(14.586553789030798,1.7305606524734458)(14.586553789030797,1.)
\psline(12.989217779919672,2.157284836846499)(14.44798573748561,2.157284836846499)
\psline(14.092534236415844,5.)(14.09253423641584,4.191420979811645)
\psline(17.005371416332736,4.05285292826653)(16.734309420906705,4.052852928266508)
\psline(17.14393946787801,5.)(17.14393946787801,4.1914209798116735)
\psline(16.040623011381836,3.76469679543864)(17.005371416332824,3.7646967954386445)
\psline(17.143939467878006,3.626128743893519)(17.143939467878006,3.1862012020317603)
\psline(15.546603458766885,3.0476331504866194)(17.005371416332824,3.0476331504866248)
\psline(16.734309420906705,2.7594770176587486)(17.005371416332736,2.7594770176587544)
\psline(17.14393946787801,2.6209089661136145)(17.14393946787801,1.5654276345489166)
\psline(16.511351863717877,2.157284836846491)(15.546603458766885,2.157284836846495)
\psline(16.27117771391438,1.8691287040186249)(16.511351863717877,1.869128704018626)
\psline(16.649919915263055,1.7305606524735067)(16.64991991526306,1.5654276345489169)
\parametricplot{-0.13000439957941268}{1.3095935835533041}{1.*0.17468417069985673*cos(t)+0.*0.17468417069985673*sin(t)+11.603821426488356|0.*0.17468417069985673*cos(t)+1.*0.17468417069985673*sin(t)+3.6487745386429355}
\parametricplot{3.0115882540111065}{4.451186237143951}{1.*0.17468417069991538*cos(t)+0.*0.17468417069991538*sin(t)+11.456222008815098|0.*0.17468417069991538*cos(t)+1.*0.17468417069991538*sin(t)+4.0843856794185935}
\psline(11.283011941344173,4.10703147416805)(11.283011941344254,5.)
\psline(11.41111109960231,3.915626782280955)(11.648932335701273,3.817533435780482)
\psline(11.777031493959207,3.626128743893361)(11.777031493959207,2.295852888391683)
\parametricplot{-0.04401778484379015}{1.395580198289112}{1.*0.16399857754465388*cos(t)+0.*0.16399857754465388*sin(t)+11.119172217131366|0.*0.16399857754465388*cos(t)+1.*0.16399857754465388*sin(t)+3.5489557614096294}
\psline(11.283011941344245,3.54173923824972)(11.28301194134425,3.1862012020318438)
\psline(10.904269739995577,3.7646967954386437)(11.147760606771742,3.7104433380662862)
\parametricplot{2.093196612489354}{2.753791322147841}{1.*0.5787740515805248*cos(t)+0.*0.5787740515805248*sin(t)+17.18571578694206|0.*0.5787740515805248*cos(t)+1.*0.5787740515805248*sin(t)+0.7811343181888133}
\parametricplot{5.234789266079132}{5.895383975737673}{1.*0.5787740515804826*cos(t)+0.*0.5787740515804826*sin(t)+16.60814359619904|0.*0.5787740515804826*cos(t)+1.*0.5787740515804826*sin(t)+1.7842933163600678}
\parametricplot{3.4564255796251633}{4.117020289283639}{1.*0.4541658971893768*cos(t)+0.*0.4541658971893768*sin(t)+17.08176270836607|0.*0.4541658971893768*cos(t)+1.*0.4541658971893768*sin(t)+1.7060635624436782}
\parametricplot{0.3148329260367089}{0.9754276356952033}{1.*0.4541658971907928*cos(t)+0.*0.4541658971907928*sin(t)+16.71209667477384|0.*0.4541658971907928*cos(t)+1.*0.4541658971907928*sin(t)+0.859364072104222}
\psline{->}(0.6385233850320104,1.)(0.7068663694415547,1.)
\psline(4.5071508110440295,2.785046961015544)(4.643836779863118,2.785046961015544)
\psline(4.5071508110440295,2.8533899454250884)(4.643836779863118,2.8533899454250884)
\psline(9.443205365579079,2.774160088842831)(9.579891334398166,2.774160088842831)
\psline(9.511548349988622,2.8425030732523746)(9.511548349988622,2.7058171044332875)
\psline(11.871668790985886,2.781534203171932)(12.008354759804973,2.781534203171932)
\psline(11.871668790985886,2.8498771875814755)(12.008354759804973,2.8498771875814755)
\psline(6.723993070993652,2.7815342031719337)(6.860679039812741,2.7815342031719337)
\psline(6.723993070993652,2.849877187581478)(6.860679039812741,2.849877187581478)
\psline(14.833690422845429,2.7829165123992325)(14.970376391664516,2.7829165123992325)
\psline(14.902033407254972,2.851259496808776)(14.902033407254972,2.714573527989689)
\rput[tl](6.972172429443146,3.089330859123315){$q^{-1}$}
\rput[tl](12.1,3.0443064633689603){$q^{-1}$}
\rput[tl](9.610602020648194,2.95){$A$}
\rput[tl](14.959500236265258,2.95){$A$}
\rput[tl](2.766893865986634,5.38557504259541){$I$}
\rput[tl](3.2171378235301575,5.376570163444539){$II$}
\rput[tl](3.7214110559789044,5.376570163444539){$1$}
\rput[tl](4.243694046729392,5.358560405142797){$2$}
\rput[tl](5.045128291156864,5.421594559198894){$I$}
\rput[tl](5.495372248700388,5.412589680048023){$II$}
\rput[tl](5.999645481149135,5.412589680048023){$1$}
\rput[tl](6.521928471899622,5.394579921746281){$2$}
\rput[tl](7.746592036418007,5.4035848008971525){$I$}
\rput[tl](8.196835993961532,5.394579921746281){$II$}
\rput[tl](8.701109226410278,5.394579921746281){$1$}
\rput[tl](9.223392217160765,5.376570163444539){$2$}
\rput[tl](10.186914286303905,5.4035848008971525){$I$}
\rput[tl](10.637158243847429,5.394579921746281){$II$}
\rput[tl](11.141431476296177,5.394579921746281){$1$}
\rput[tl](11.663714467046663,5.376570163444539){$2$}
\rput[tl](13.005441460526365,5.394579921746281){$I$}
\rput[tl](13.455685418069889,5.38557504259541){$II$}
\rput[tl](13.959958650518635,5.38557504259541){$1$}
\rput[tl](14.482241641269123,5.3675652842936685){$2$}
\rput[tl](15.57183201852445,5.412589680048023){$I$}
\rput[tl](16.022075976067974,5.4035848008971525){$II$}
\rput[tl](16.52634920851672,5.4035848008971525){$1$}
\rput[tl](17.048632199267207,5.38557504259541){$2$}
\rput[tl](13.689812275992521,0.7840817965003486){$\mathcal{T}_1w'$}
\psline[linestyle=dashed,dash=1pt 1pt](12.69927556939677,3.2334089255372502)(12.69927556939677,0.8201013131038328)
\psline[linestyle=dashed,dash=1pt 1pt](12.69927556939677,3.2334089255372502)(14.743383136644365,3.2334089255372502)
\psline[linestyle=dashed,dash=1pt 1pt](14.743383136644365,3.2334089255372502)(14.743383136644365,0.8201013131038328)
\psline[linestyle=dashed,dash=1pt 1pt](12.69927556939677,0.8201013131038328)(14.743383136644365,0.8201013131038328)
\rput[tl](14.797412411549589,1.0722379293282194){$w'$}
\rput[tl](16.040085734369715,0.70303788414251){$wg_1$}
\pscircle[linestyle=dashed,dash=1pt 1pt](6.112286245550237,3.7646967954386428){0.16877901128729533}
\end{pspicture*}
\caption{A typical example of the recursion phenomenon discussded in Lemma \ref{lemma_4}.}
\label{figure_last}
\end{figure}

\begin{thm} \label{theorem_2}

$\mathrm{(Part \ A)}$ Let $w\in B_{2,n}$ be a product of $k$ in number loopings $\mathcal{T}^{\pm 1}_i,\tau_i^{\pm 1}, i={1,2,\ldots,n}$. Call $m,M$ the minimum and maximum index of the loopings in this product. Then as an element of $\mathrm{H}_{2,n}(q)$, $w$ can be written as a finite  $\mathbb{Z}[q^{\pm 1}]$-linear combination of the form (suppressing the coefficient in $\mathbb{Z}[q^{\pm 1}]$ of each term)
\[ w=\sum w'G \]
so that:
\begin{enumerate}
\item[i)] each $w'$ is the image of a braid in $B_{2,n}$ which is written as a product of $k$ in number loopings $\mathcal{T}^{\pm 1}_i,\tau_i^{\pm 1}, i={1,2,\ldots,n}$. 
\item[ii)] the index of each looping in $w'$ lies in the interval $[m,M]$.
\item[iii)] the indices of the loopings in $w'$ increase from left to right.
\item[iv)] $G$ is a finite product of $g_i^{\pm 1}$'s, $i\in [m,M-1]$ if $m<M$, or  $G=1$ if $m=M$.
\end{enumerate}

\noindent $\mathrm{(Part \ B)}$
Any element in $\mathrm{H}_{2,n}(q)$ can be written as a finite  $\mathbb{Z}[q^{\pm 1}]$-linear combination of the form (suppressing the coefficient in $\mathbb{Z}[q^{\pm 1}]$ of each term): 
\[ \sum \Pi_1\Pi_2\cdots\Pi_n G\]
\noindent where $ G$ is a finite product of the braiding generators $g_1,\ldots,g_{n-1}$, and     $\Pi_i$ is  a finite product of only the loopings  $\mathcal{T}_i,\tau_i,\mathcal{T}_i^{-1},\tau_i^{-1}$ for all $i$. Thus 
$
\Lambda_n := \left\{ \Pi_1\Pi_2\cdots\Pi_n G \ | \ G=\right.$ finite product of  braiding generators, and $  \Pi_i=$ finite product of only the loopings $ \mathcal{T}_i,\tau_i,\mathcal{T}_i^{-1},\tau_i^{-1}, \forall i \left. \right\}
$
is a spanning set of the algebra $\mathrm{H}_{2,n}(q)$.
\end{thm}

Part B of the Theorem follows immediately from Part A, since any element in $\mathrm{H}_{2,n}(q)$ can be written as a finite  $\mathbb{Z}[q^{\pm 1}]$-linear combination of images of braids of $B_{2,n}$, and as we already mentioned above,  each such image is written itself as a finite  $\mathbb{Z}[q^{\pm 1}]$-linear combination of products of $\mathcal{T}_i,\tau_i $'s followed by $g_i$'s because of Lemma \ref{lemma_2} below. We give the proof of Part A at the end, after proving the necessary lemmata that follow.

\begin{lemma}\label{lemma_2} Let $A=q^{-1}-1, B=q-1$, and let $1$ be the identity element of $\mathrm{H}_{2,n}(q)$.  Then the following hold in $\mathrm{H}_{2,n}(q)$: \vspace{1ex}

\begin{tabular}{llll}
(1) & $g_i^{-1}=q^{-1}g_i+A\cdot 1$  & &\vspace{1ex}  \\ 
(2) &$g_i\mathcal{T}_j^{\pm 1}=\mathcal{T}_j^{\pm 1}g_i$ & $g_i\tau_j^{\pm 1}=\tau_j^{\pm 1}g_i$& whenever $j\neq i, i+1$ \vspace{1ex}\\
(3) & $g_i\mathcal{T}_i=q^{-1}\mathcal{T}_{i+1}g_i+A\mathcal{T}_{i+1}$ & $g _i\tau_i=q^{-1}\tau_{i+1}g_i+A\tau_{i+1}$& \vspace{1ex} \\
(4) & $g_i\mathcal{T}_i^{-1}=\mathcal{T}_{i+1}^{-1} - A \mathcal{T}_{i+1}^{-1} + A \mathcal{T}_i^{-1}$ & $g_i\tau_i^{-1} = \tau_{i+1}^{-1} - A \tau_{i+1}^{-1} + A \tau_i^{-1}$ & \vspace{1ex}\\
(5) & $g_i \mathcal{T}_{i+1} = q \mathcal{T}_i g_i + B \mathcal{T}_{i+1}$ & $g_i \tau_{i+1} = q \tau_i g_i + B \tau_{i+1}$ & \vspace{1ex}\\
(6) & $g_i \mathcal{T}_{i+1}^{-1} = q^{-1}\mathcal{T}_i^{-1}g_i + A \mathcal{T}_i^{-1}$ & $g_i \tau_{i+1}^{-1} = q^{-1}\tau_i^{-1}g_i + A \tau_i^{-1}$.& 
\end{tabular} \vspace{1ex}

\hspace{0.2ex} (7) (The passage property) Any product $g_k^\epsilon  t_l^\zeta \ (\epsilon, \zeta\in \{ 1,-1\}$, $t_l$ a looping) can be written as a finite linear combination of the form (suppressing the coefficient in $\mathbb{Z}[q^{\pm 1}]$ of each term on the right-hand side):

\[ g_i^\epsilon  t_l^\zeta = \sum t_{l}^{\zeta} g_{i}^{\epsilon} + \sum t_{i}^{\zeta} + \sum t_{i}^{\zeta}g_i^{\epsilon} + \sum t_{i+1}^{\zeta}.\]

\noindent (where possibly some of the terms are missing). \vspace{1ex}

\hspace{0.2ex} (8) (The big passage property) Let $\Pi$ be a finite product of $k$ in number loopings with indices in the interval $[m,M]$, and let $i\in[m,M-1]$. Then $g_i^{\pm 1} \Pi$ can be written as a finite linear combination of the form (suppressing the coefficient in $\mathbb{Z}[q^{\pm 1}]$ of each term on the right-hand side):
\[ g_i^{\pm 1} \Pi = \sum\Pi_1 g_i^{\pm 1}  + \sum\Pi_2 \]
with each $\Pi_1,\Pi_2$ a product of  $k$ in number loopings with indices in $[m,M]$ (and where possibly some terms are missing).
\end{lemma}

\begin{proof}  The proof is immediate, as  item (2) can be seen in the braid level via trivial braid isotopies, whereas items (2)--(6) can be seen pictorially after at most two applications of the quadratic relation to the braids of the left-hand side. Item (7) is verified by checking the right-hand side of the equalities in items (2)--(6) whenever $\epsilon=1$, and using this result and item (1) whenever $\epsilon=-1$. The important fact in the passage property is that, during the passage of $g_i$ to the right the index of the looping either does not change at all or, if it does, it decreases by $1$ but then never below the index $i$ of the braiding generator $g_i$, or else it increases by $1$ but then never by $1$ above the index $i$ of the braiding generator. This remark and the fact that each monomial on both sides of the equalities in items (2)--(6) contains a single looping, settle immediately item (8).
\end{proof}   

\begin{lemma}\label{lemma_3}
For $j<i$ and $\epsilon,\zeta\in \{ 1,-1\}$ each one of the words $\mathcal{T}_i^\epsilon \mathcal{T}_j^\zeta, \mathcal{T}_i^\epsilon \tau_j^\zeta, \tau_i^\epsilon \mathcal{T}_j^\zeta,\tau_i^\epsilon \tau_j^\zeta$ can be written as required in Theorem \ref{theorem_2} as a linear combination of the form  (suppressing the coefficient in $\mathbb{Z}[q^{\pm 1}]$ of each term on the right-hand side):
\begin{enumerate}
\item $\mathcal{T}_i^\epsilon \mathcal{T}_j^\zeta=\mathcal{T}_j^\zeta \mathcal{T}_i^\epsilon, \ \tau_i^\epsilon \tau_j^\zeta =\tau_j^\zeta \tau_i^\epsilon, \ \mathcal{T}_i^\epsilon \tau_j^\zeta = \tau_j^\zeta \mathcal{T}_i^\epsilon $ \vspace{1ex} 
\item $\tau_i^{\epsilon} \mathcal{T}_j^{\epsilon} = \mathcal{T}_j^{\epsilon} \tau_i^{\epsilon} + \mathcal{T}_j^{\epsilon} \tau_i^{\epsilon} G^{\epsilon} + \tau_j^{\epsilon} \mathcal{T}_{i}^{\epsilon} G^{\epsilon}$,  where  $ G=g_j g_{j+1} \ldots g_{i-2}g_{i-1}^{-1}g_{i-2}^{-1}\ldots g_{j+1}^{-1}g_j^{-1}  $. \vspace{1ex} 
\item $\tau_i^\epsilon \mathcal{T}_j^{-\epsilon} = \mathcal{T}_j^{-\epsilon} \tau_i^\epsilon + \mathcal{T}_j^{-\epsilon} \tau_j^\epsilon G^\epsilon + \tau_i^\epsilon \mathcal{T}_i^{-^\epsilon}G^\epsilon $, where $ G=g_jg_{j+1}\ldots g_{i-2}g_{i-1}g_{i-2}\ldots g_{j+1} g_j.$
\end{enumerate}
\end{lemma}

\begin{proof} The proof of this lemma is also immediate as part (1) can be seen at the braid level as braid isotopies, whereas the other two parts can also be seen pictorially after two applications of the quadratic relation to the braids of the left-hand side (the obvious ones, so that the $i$-looping  can be moved above the $j$-looping). 

The loopings appearing on the right-hand side of the equalities given in Lemma \ref{lemma_3} have the same indices as those on the left-hand side.  Nevertheless,  one of the monomials on the right-hand side in case (3) still starts with an $i$-looping instead of a $j$-looping, and this is the reason that recursion can occur in trying to put all words of $\mathrm{H}_{2,n}(q)$ in the required form. The following lemma deals with the heart of these recursion phenomena.  For both the statement and its proof we use the notational convention that  $[i,j]$ in the bottom of a product of looping or braiding generators indicates that their indices lie in the interval $[i,j]$, whereas $<i,j>$ indicates that these indices are also in increasing order (from left to right). 
\end{proof} 

\begin{lemma}\label{lemma_4} Let us denote elements in $\{ \mathcal{T}_i^{\pm 1}, \tau_i^{\pm 1} \}$ indiscreetly by $t_i$. Then each one of the words $\tau_M^\epsilon \mathcal{T}_M^{-\epsilon} t_m^{\ \zeta}$ with $m<M$ and $\epsilon, \zeta \in \{-1,1 \}$ can be written as required in Theorem \ref{theorem_2} as a finite linear combination of the form (suppressing the coefficient in $\mathbb{Z}[q^{\pm 1}]$ of each term on the right-hand side):
\begin{eqnarray*}
\tau_M^\epsilon \mathcal{T}_M^{-\epsilon} t_m^{\ \zeta} =\sum \left( \underset{<m,M>}{t_m t_{m_1} t_{m_2}} \right)  \underset{[m,M-1]}{G}
\end{eqnarray*}

\noindent where each $G$ is a finite product of $g_i^{\pm 1}$'s (notice the crucial fact that every term of the last sum starts with an $m$-looping). 
\end{lemma}

\begin{proof}  We examine all possible cases.  Using the notational convention introduced just before the statement, and underlining the part of the word on which we apply a previous lemma, the quadratic relation, or equalities coming from isotopies in the braid level, we have:
 
\begin{enumerate}
\item  $\tau_M^\epsilon \underline{\mathcal{T}_M^{-\epsilon} \tau_m^\zeta} \overset{Lem. \ \ref{lemma_3}}{=}   \underline{ \tau_M^\epsilon  \tau_m^\zeta} \mathcal{T}_M^{-\epsilon} \overset{Lem. \ \ref{lemma_3}}{=}  \tau_m^\zeta \tau_M^\epsilon \mathcal{T}_M^{-\epsilon}.
$

\item $
 \tau_M^\epsilon \underline{ \mathcal{T}_M^{-\epsilon}  \mathcal{T}_m^\epsilon } \overset{Lem. \ \ref{lemma_3}}{=} \underline{\tau_M^\epsilon \mathcal{T}_m^\epsilon} \mathcal{T}_M^{-\epsilon}   \overset{Lem. \ \ref{lemma_3}}{=} \sum t_mt_M\underset{[m,M-1]}{\underline{G}} \underline{\mathcal{T}_M}^{-\epsilon} \overset{Lem. \ \ref{lemma_2}}{=} \sum t_m \underline{ t_M} \left( \underset{[m,M]}{\underline{ t_{m'}}} \right)\underset{[m,M-1]}{G_1} \\ 
\overset{Lem. \ \ref{lemma_3}}{=} \sum t_m \left( \underset{<m,M>}{t_{m_1} t_{m_2}}\right) \underset{[m,M-1]}{G_2} \underset{[m,M-1]}{G_1} = \sum  \left( \underset{<m,M>}{t_m  t_{m_1} t_{m_2}} \right)  \underset{[m,M-1]}{G_3}.
$

\item Finally for $\tau_M^{-1}  \mathcal{T}_M  \mathcal{T}_m$ (and similarly for $\tau_M  \mathcal{T}_M^{-1}  \mathcal{T}_m^{-1}$): Observe that by Lemma \ref{lemma_3} it is $\tau_M^{-1}  \mathcal{T}_M  \mathcal{T}_m=\tau_M^{-1} \mathcal{T}_m \mathcal{T}_M$. In this form of the braid, apply the quadratic relation (Q.R.) at the second crossing of $\tau_M^{-1}$ with $ \mathcal{T}_m$, perform at the braid level the obvious isotopies (isot.) and put as before $A=q^{-1}-1$ to obtain:

\small
\begin{eqnarray*}
\tau_M^{-1}  \underline{ \mathcal{T}_M  \mathcal{T}_m} &\overset{Lem. \ \ref{lemma_3}}{=}& \tau_M^{-1} \mathcal{T}_m \mathcal{T}_M \\
&\overset{(Q.R.),(isot.)}{=}& q^{-1}\tau_m \left( \underset{[m,M-1]}{\underline{G_1}}\right) \underline{\tau_{M-1}^{-1}} \left( \underset{[m,M-1]}{G_2}\right)  \mathcal{T}_M + A \tau_M^{-1} \mathcal{T}_M   \mathcal{T}_m G \\
& \overset{Lem. \ \ref{lemma_2}}{=} & q^{-1}\tau_m \left( \underset{[m,M]}{t_k} \right)\left( \underset{[m,M-1]}{\underline{G_3}}\right) \left( \underset{[m,M-1]}{\underline{G_2}}\right)\mathcal{T}_M +  A \tau_M^{-1} \mathcal{T}_M   \mathcal{T}_m G
\end{eqnarray*}\normalsize

\noindent where $G=g_mg_{m+1}\cdots g_{M-2}g_{M-1}^{-1}g_{M-2}^{-1}g_{m}^{-1}$. Thus (suppressing coefficients in the big sums):

\small
\begin{eqnarray*}
\tau_M^{-1}  \mathcal{T}_M  \mathcal{T}_m (1-A G) &=& q^{-1}\tau_m \left( \underset{[m,M]}{t_k} \right) \left( \underset{[m,M-1]}{\underline{G_3 G_2}}\right)  \underline{\mathcal{T}_M}\\
& \overset{Lem. \ \ref{lemma_2}}{=}& \sum \tau_m \left( \underset{[m,M]}{\underline{t_k}} \right) \left( \underset{[m,M]}{\underline{t_l}} \right) \underset{[m,M-1]}{G_4} \\
&\overset{Lem. \ref{lemma_3}}{=}& \sum \tau_m \left( \underset{<m,M>}{\underline{t_{m_1} t_{m_2}}} \right) \underset{[m,M-1]}{G_4} = \sum \left( \underset{<m,M>}{ \tau_m t_{m_1} t_{m_2}} \right) \underset{[m,M-1]}{G_4}   
\end{eqnarray*}
\normalsize

\noindent But $1-AG=q^{-1}G^2$ is reversible. Indeed, if  we apply for the braid $G^{-1}$ the quadratic relation at the crossing of its non straight strands and perform the obvious isotopies we obtain:

\begin{eqnarray*}
G^{-1}  =  q^{-1} G + A \cdot 1 & \Leftrightarrow & -A \cdot  1 = -G^{-1} + q^{-1} G \Rightarrow -A G = -1 + q^{-1} G^2 \\ & \Leftrightarrow &  1-A G = q^{-1}G^2.
\end{eqnarray*}

\noindent This means that $1-AG=q^{-1}G^2$ is a reversible element of $\mathrm{H}_{2,n}(q)$ as claimed, with inverse $(1-AG)^{-1}=qG^{-2}$. So the last equality for $\tau_M^{-1}  \mathcal{T}_M  \mathcal{T}_m (1-A G)$ above now gives:

\begin{eqnarray*}
\tau_M^{-1}  \mathcal{T}_M  \mathcal{T}_m  = \sum \left( \underset{<m,M>}{ \tau_m t_{m_1} t_{m_2}} \right) \underset{[m,M-1]}{G_4} (q G^{-2})   =  \sum \left( \underset{<m,M>}{ \tau_m t_{m_1} t_{m_2}} \right) \underset{[m,M-1]}{G_5},
\end{eqnarray*}

\noindent where we suppressed coefficients in the last sum, and the lemma is established in all cases.
\end{enumerate}
\end{proof}

We are now ready to prove Part A of Theorem \ref{theorem_2}, which is all that remains.

\begin{proof} (Part A of Theorem \ref{theorem_2}) For a  braid word $w$ given in the form of the product of $k(w)\geq 1$ in number  loopings, let $m(w)$ and $M(w)$ be respectively the minimum and maximum index of the loopings in this product, and let $(k(w),d(w))$ be the \textit{index pair} of $w$, where $d(w)=M(w)-m(w)$. Whenever $k(w)=1$, then necessarily $d(w)=0$. Otherwise $k(w)\geq 2$ and $d(w)$ can be any number in the interval $[0,n-1]$. Let $A$ be the set of all possible pairs $(k(w),d(w))$. We totally order the elements of $A$ via the relation $<$ defined as:
\[ (k,d)<(k',d') \overset{def}{\Leftrightarrow} [ (k<k')  \text{\ or\ } (k=k' \text{\ and \ } d<d')]. \]
The smallest element of $A$ in this ordering is $(1,0)$, and the next ones are $(2,0),(2,1),\ldots,(2,n-1)$. The pairs $(3,0),(3,1),\ldots,(3,n-1)$ follow, and then the pairs $(4,0),(4,1),\ldots,(4,n-1)$  and so on. 

\smallbreak
We prove the Theorem by induction on the pair $(k(w),d(w))\in A$ with the above ordering:

For  $w$ so that $(k(w),d(w))=(1,0)$, the result is true since $w$ consists of a unique looping which is thus in the required form without any further arrangement. Also, for $w$ so that $(k(w),d(w))=(2,d)$, it is  either $w=t_1t_2$ for $t_1,t_2\in \{ \mathcal{T}_i^{\pm 1}, \tau_i^{\pm 1}\}$ for some $ i=1,2,\ldots,n $ and the result is immediate (here $G=1$), or else $w=t_1t_2$ for $t_1,t_2\in \{ \mathcal{T}_i^{\pm 1}, \tau_i^{\pm 1} | i=1,2,\ldots,n \} $ where $t_1, t_2$ are not of the same index, and the result is also true by Lemma \ref{lemma_3}.

\smallbreak
Our Induction Hypothesis (I.H.): Let's assume that for $(2,n-1)\leq(k_0,d_0)$, the Theorem is true for all  $w'$ which are products of loopings with index pair $(k(w'),d(w'))\leq(k_0,d_0)$.
\smallbreak

Let us now consider a word $w$ for which its index pair $(k(w),d(w))$ follows immediately after $(k_0,d_0)$ in the ordering of the index pairs in A. Thus the Theorem holds for all products of loopings with index pair smaller than that of $w$; by the way, it is $3\leq k(w)$. We'll show that the Theorem holds for $w$ as well, and we'll have finished. \smallbreak

Notational conventions: we denote $k(w)=k, \ m(w)=m,\ M(w)=M$, we call as \textit{settled} all products of $k$ loopings for which the Theorem holds, and we put information regarding the indices in the bottom of products, with $[i,j]$ indicating indices lying in the interval $[i,j]$, whereas $<i,j>$ indicating that these indices also increase from left to right. Underlined products indicate the parts of a word on which we apply a lemma or the Induction Hypothesis.  All sums are finite, and each term in them always has $k(w)$ in number loopings because of the lemmata and the Induction Hypothesis. We ignore the coefficients in $\mathbb{Z}[q^{\pm 1}]$ of each term in the sums, whereas by $G$ we always denote a finite product of braiding generators. Finally, by $t_i$ we denote any looping of index $i$. Let's return to the proof:

\smallbreak
For $m=M$ the result is immediate since $w$ is a product of loopings all of the same index. 

\smallbreak
For $m<M$, we will expedite the arguments by making the observation (Obs) that the result holds for any $ w=t_m\ldots$, a fact whose truth is almost immediate:

\smallbreak
$t_m \underline{\ldots} \overset{I.H.}{=}\sum t_m \left( \underset{<m,M>}{\cdots} \right) = \sum  \left( \underset{<m,M>}{t_m \cdots} \right). \hspace{1cm} (Obs) $

\smallbreak
Now for an arbitrary $w$ for which $m<M$ we have:

\smallbreak
$w= \left\{ \begin{array}{ll} 
\underset{(M'<M)}{t_{M'}}\left( \underset{[m,M] }{\cdots} \right) & (I) \ \text{or} \vspace{1ex} \\
t_M  \left( \underset{[m,M]}{\cdots}  \right)& (II) 
\end{array} \right.  $

\smallbreak
For type (I) words:

\smallbreak
$\underset{(M'<M)}{t_{M'}}\left( \underset{[m,M] }{\underline{\cdots}} \right) \overset{I.H.}{=}\sum \underset{(M'<M)}{t_{M'}}\left( \underset{(M''<M) \ <m,M> }{\cdots t_{M''}}  \right) \underset{[m,M-1]}{G} + \sum \underset{(M'<M)}{t_{M'}}\left( \underset{<m,M> }{\cdots t_{M}}   \right) \underset{[m,M-1]}{G} . $

\smallbreak
For a term $w'$ in the first sum (ignoring the $G$'s) it is $k(w')=k(w), d(w')<d(w)$, so $w'=$settled.

\smallbreak
For a term in the second sum (ignoring the $G$'s):

\smallbreak
$w'= \underset{(M\leq M-1)}{\underline{t_{M'}}}\left( \underset{<m,M-1> }{\underline{\cdots} } \right)(t_M\cdots t_M)   \overset{I.H.}{=} \sum \left( \underset{<m,M-1>}{\cdots}\underset{[M,M-2]}{G}\right)(t_M\cdots t_M)\overset{Lemma \ 1}{=} $ 

 $
=\sum\left( \underset{<m,M-1>}{\cdots}\right) (t_M\cdots t_M) \underset{[m,M-2]}{G}  $ and each term in this sum is settled.

\smallbreak
For type (II) words:

\smallbreak
$t_M  \left( \underset{[m,M]}{\underline{\cdots}}  \right) \overset{I.H.}{=} \sum t_M\left( \underset{<m,M>}{\cdots} \right) \underset{[m,M-1]}{G} = \sum t_M\left( \underset{(m<m')\ <m,M>}{t_{m'}\cdots} \right) \underset{[m,M-1]}{G} + \sum t_M\left( \underset{<m,M>}{t_m\cdots} \right) \underset{[m,M-1]}{G}$

\smallbreak
For a term $w'$ in the first sum (ignoring the $G$'s) it is $k(w')=k(w), d(w')<d(w)$, so $w'=$settled.

\smallbreak
For a term in the second sum (ignoring the $G$'s):

If the product of the first two letters is not $\tau^{-1}_M \mathcal{T}_m$ or 
$\tau_M \mathcal{T}^{-1}_m$ then

\smallbreak
$\underline{t_M} \left( \underset{<m,M>}{ \underline{t_m }\cdots} \right) \overset{Lemma \ \ref{lemma_3}}{=} \sum t_m t_M \underset{[m,M-1]}{\underline{G}} \left( \underset{<m,M>}{\underline{ \cdots}} \right) \overset{Lemma \ \ref{lemma_2}}{=} \sum t_m t_M \left( \underset{[m,M]}{ \cdots} \right)\underset{[m,M-1]}{G'}  $

\smallbreak
and each term in these sums is settled by (Obs).

\smallbreak
If the product of the first two letters is  $\tau^{-1}_M \mathcal{T}_m$ or 
$\tau_M \mathcal{T}^{-1}_m$ then in the first case (and similarly in the second):

\begin{eqnarray*}
\underline{\tau_M}^{-1} \left( \underset{<m,M>}{ \underline{ \mathcal{T}_m} \cdots} \right) &\overset{Lemma \ \ref{lemma_3}}{=}& \sum t_m t_M \underset{[m,M-1]}{\underline{G}} \left( \underset{<m,M>}{ \underline{\cdots}} \right) +  \sum \tau^{-1}_M \mathcal{T}_M\underset{[m,M-1]}{\underline{G}} \left( \underset{<m,M>}{ \underline{\cdots}} \right)\\
&\overset{Lemma \ \ref{lemma_2}}{=}& \sum t_m t_M \left( \underset{[m,M]}{ \cdots} \right)\underset{[m,M-1]}{G'} + \sum \tau^{-1}_M \mathcal{T}_M \left( \underset{[m,M]}{ \cdots} \right) \underset{[m,M-1]}{G'} 
\end{eqnarray*}

and each term in the first sum is settled by (Obs).

\smallbreak
For a term in the second sum (ignoring the $G$'s):

\begin{eqnarray*}  \sum \tau^{-1}_M \mathcal{T}_M \left( \underset{[m,M]}{ \underline{\cdots}} \right)  &\overset{I.H.}{=}& \sum \tau^{-1}_M \mathcal{T}_M \left( \sum \underset{<m,M>}{ \cdots}  \underset{[m,M-1]}{G} \right)\\
&=& \sum \tau^{-1}_M \mathcal{T}_M \left( \underset{(m<m') \ [m,M]}{ t_{m'}\cdots} \right) \underset{[m,M-1]}{G} + \sum \tau^{-1}_M \mathcal{T}_M \left( \underset{[m,M]}{t_m\cdots} \right) \underset{[m,M-1]}{G}  
\end{eqnarray*}

Any term $w'$ in the first sum (ignoring the $G$'s) has $k(w')=k(w), d(w')<d(w)$, so $w'=$settled.

\smallbreak
For a term in the second sum (ignoring the $G$'s):

\smallbreak
$\underline{\tau^{-1}_M \mathcal{T}_M} \left( \underset{[m,M]}{ \underline{t_m} \cdots} \right) \overset{Lemma \ \ref{lemma_4}}{=} \sum \left( \underset{[m,M]}{ t_m t_{m'} t_{m''}} \right) \underset{[m,M-1]}{\underline{G}} \left( \underset{[m,M] }{\underline{\cdots}} \right) \overset{Lemma \ \ref{lemma_2} }{=} \sum \left( \underset{[m,M]}{ t_m t_{m'} t_{m''}} \right) \left( \underset{[m,M] }{\cdots} \right) \underset{[m,M-1]}{G}  $

\smallbreak
and each term in this sum (ignoring the $G$'s) is settled by (Obs).

\smallbreak
Summing up, $w$ is always written as $\sum \underset{<m,M>}{\cdots} \ \underset{[m,M-1]}{G}$ as wanted, and we are done.
\end{proof}

\begin{conj}
The set $\Lambda_n$ is a $\mathbb{Z}[q^{\pm 1}]$-linear basis for the  algebra $\mathrm{H}_{2,n}(q)$.
\end{conj}

\smallbreak

\noindent \textbf{Conclusion and further research.}
In this paper we have given a spanning set $\Lambda_n$ for the mixed Hecke algebra $\mathrm{H}_{2,n}(q)$ which is defined as the quotient of the group algebra  $\mathbb{Z}[q^{\pm 1}]B_{2,n}$ of the mixed braid group $B_{2,n}$ with two fixed strands, over the quadratic  relations of the usual Iwahori--Hecke algebra. The algebra $\mathrm{H}_{2,n}(q)$ is related to the knot theory of various $3$-manifolds and, as stated in the conjecture above, we believe that the set $\Lambda_n$ actually is a  linear basis for this algebra, a fact that promises to lead to the construction of knot invariants in the above $3$-manifolds via the braid-theoretic method. The proof of the conjecture seems to require tools from representation theory and it is the subject of ongoing research.


\end{document}